\documentclass[11pt]{amsart}

\usepackage{amsmath, amssymb, amsthm, xypic, stmaryrd,graphicx,mathtools}
\usepackage[all]{xy}

\numberwithin{equation}{section}
\DeclareMathOperator{\supp}{supp}

\DeclareMathOperator{\op}{op}
\DeclareMathOperator{\vol}{vol}

\DeclareMathOperator{\grad}{grad}

\DeclareMathOperator{\ind}{index}

\DeclareMathOperator{\End}{End}

\DeclareMathOperator{\Spin}{Spin}
\DeclareMathOperator{\SO}{SO}
\DeclareMathOperator{\SU}{SU}

\DeclareMathOperator{\DInd}{D-Ind}

\DeclareMathOperator{\Ad}{Ad}

\DeclareMathOperator{\Cl}{Cl}

\newcommand{\beq}[1]{\begin{equation} \label{#1}}
\newcommand{\eeq}{\end{equation}}

\newcommand{\bea}{\begin{eqnarray}}
\newcommand{\eea}{\end{eqnarray}}

\begin{document}

\theoremstyle{plain}
\newtheorem{theorem}{Theorem}[section]
\newtheorem{thm}{Theorem}[section]
\newtheorem{lemma}[theorem]{Lemma}
\newtheorem{proposition}[theorem]{Proposition}
\newtheorem{prop}[theorem]{Proposition}
\newtheorem{corollary}[theorem]{Corollary}
\newtheorem{conjecture}[theorem]{Conjecture}
\newtheorem{question}[theorem]{Question}

\theoremstyle{definition}
\newtheorem{definition}[theorem]{Definition}
\newtheorem{defn}[theorem]{Definition}
\newtheorem{example}[theorem]{Example}
\newtheorem{remark}[theorem]{Remark}
\newtheorem{rem}[theorem]{Remark}

\newcommand{\C}{\mathbb{C}}
\newcommand{\R}{\mathbb{R}}
\newcommand{\Z}{\mathbb{Z}}
\newcommand{\N}{\mathbb{N}}
\newcommand{\Q}{\mathbb{Q}}

\newcommand{\Supp}{{\rm Supp}}

\newcommand{\field}[1]{\mathbb{#1}}
\newcommand{\bZ}{\field{Z}}
\newcommand{\bR}{\field{R}}
\newcommand{\bC}{\field{C}}
\newcommand{\bN}{\field{N}}
\newcommand{\bT}{\field{T}}
\newcommand{\cB}{{\mathcal{B} }}
\newcommand{\cK}{{\mathcal{K} }}
\newcommand{\cF}{{\mathcal{F} }}
\newcommand{\cO}{{\mathcal{O} }}
\newcommand{\cE}{\mathcal{E}}
\newcommand{\cS}{\mathcal{S}}
\newcommand{\cN}{\mathcal{N}}
\newcommand{\calL}{\mathcal{L}}

\newcommand{\KK}{K \! K}

\newcommand{\norm}[1]{\| #1\|}

\newcommand{\Spinc}{\Spin^c}

\newcommand{\HH}{{\mathcal{H} }}
\newcommand{\Hpi}{\HH_{\pi}}

\newcommand{\DNR}{D_{N \times \R}}


\def\kt{\mathfrak{t}}
\def\kk{\mathfrak{k}}
\def\kp{\mathfrak{p}}
\def\kg{\mathfrak{g}}
\def\kh{\mathfrak{h}}
\def\so{\mathfrak{so}}
\def\cut{c}

\newcommand{\ddt}{\left. \frac{d}{dt}\right|_{t=0}}

\title[Positive scalar curvature  and a Callias index theorem]{Positive scalar curvature and an equivariant Callias-type index theorem for proper actions}
\author{Hao Guo}
\address{Texas A\&M University}
\email{haoguo@math.tamu.edu}
\author{Peter Hochs}
\address{University of Adelaide}
\email{peter.hochs@adelaide.edu.au}
\author{Varghese Mathai}
\address{University of Adelaide}
\email{mathai.varghese@adelaide.edu.au}

\maketitle

\begin{abstract}
For a proper action by a locally compact group $G$ on a manifold $M$ with a $G$-equivariant $\Spin$-structure, we obtain obstructions to the existence of complete $G$-invariant Riemannian metrics with uniformly positive scalar curvature. We focus on the case where $M/G$ is noncompact. The obstructions follow from a Callias-type index theorem, and relate to positive scalar curvature near hypersurfaces in $M$. We also deduce some other applications of this index theorem. If $G$ is a connected Lie group, then the obstructions to positive scalar curvature vanish under a mild assumption on the action. In that case, we generalise a construction by Lawson and Yau to obtain complete $G$-invariant Riemannian metrics with uniformly positive scalar curvature, under an equivariant bounded geometry assumption.
\end{abstract}

\tableofcontents

\section{Introduction}

Let $G$ be a locally compact group, acting properly on a manifold $M$.
Suppose that $M$ has a $G$-equivariant $\Spin$-structure.
The results in this paper are about the following question.
\begin{question}\label{question psc}
When does $M$ admit a complete, $G$-invariant Riemannian metric with uniformly positive scalar curvature?
\end{question}
We are mainly interested in the case where $M/G$ is noncompact. 

The literature on the non-equivariant case of Question \ref{question psc} is too vast to summarise, but important work where $M$ is noncompact was done by Gromov and Lawson \cite{Gromov83}. A more refined perspective on the non-equivariant case is to consider a manifold $X$, and let $M$ be its universal cover, and $G$ is fundamental group. This allows one to construct obstructions to metrics of positive scalar curvature in terms of $G$-equivariant index theory on $M$, refining index-theoretic obstructions on $X$. If $X$ is compact, this is the origin of the even more refined Rosenberg index 
\cite{Rosenberg86a, Rosenberg83,  Rosenberg86b}, in $KO$-theory of the real maximal group $C^*$-algebra of $G$.

More generally, if $G$ is a discrete group not necessarily acting freely on $M$, then $X \coloneqq M/G$ is an orbifold, whence Question \ref{question psc} becomes the question of whether $X$ admits an orbifold metric of positive scalar curvature.

We consider the case where $G$ is not necessarily discrete, and does not necessarily act freely. Results on this case of Question \ref{question psc} in the case where $G$ is an almost-connected Lie group and $M/G$ is compact were obtained in \cite{GMW, Piazza18}. In this paper, we focus on the case where $M/G$ is noncompact. 

\subsection*{Results on positive scalar curvature}

We first obtain obstructions to $G$-invariant Riemannian metrics with positive scalar curvature, both in the $K$-theory of the maximal or reduced group $C^*$-algebra of $G$, and in terms of numerical topological invariants generalising the $\hat A$-genus. If $G$ is a connected Lie group, then these obstructions vanish under a mild assumption, as shown in \cite{HM16}. In that case, we construct  $G$-invariant Riemannian metrics with uniformly positive scalar curvature, under an equivariant bounded geometry assumption.

Our most general obstruction result is the following.
\begin{theorem}\label{thm obstr psc intro}
Let $H\subset M$ be a $G$-invariant, cocompact hypersurface with trivial normal bundle,  that partitions $M$ into two open sets. If $M$ admits a complete, $G$-invariant Riemannian metric with nonnegative scalar curvature, and positive scalar curvature in a neighbourhood of $H$, then 
\[
\ind_G(D^H) = 0 \quad \in K_*(C^*(G)),
\]
for a $\Spin$-Dirac operator $D^H$ on $H$.
\end{theorem}
See Theorem \ref{thm obstr psc}. In the case where $M$ is the universal cover of a manifold $X$ and $G$ is its fundamental group, this becomes Theorem A in \cite{Cecchini16}.

Theorem \ref{thm obstr psc intro} implies topological obstructions to $G$-invariant Riemannian metrics with positive scalar curvature. Let $g \in G$, and let $H^g \subset H$ be the fixed point set of $g$. Let $\cN \to H^g$ be its normal bundle in $H$, and let $R^{\cN}$ be the curvature of the Levi--Civita connection restricted to $\cN$. The \emph{$g$-localised $\hat A$-genus} of $H$ is
\[
\hat A_g(H) \coloneqq \int_{H^g} c^g \frac{\hat A(H^g) }{\det(1-g e^{-R^{\mathcal{N}}/2\pi i})^{1/2}},
\]
for a cutoff function $c^g$ on $X^g$. If $G$ acts freely, then $\hat A_g(H) = 0$ if $g\neq e$, and $\hat A_e(H)$ is the $\hat A$-genus of $H/G$. In general, for example in the orbifold case, $\hat A_g(H)$ may be nonzero for different $g$.
\begin{corollary} \label{cor obstr psc intro}
Consider the setting of Theorem \ref{thm obstr psc intro}.
Suppose that either
\begin{itemize}
\item $G$ is any locally compact group and $g = e$;
\item $G$ is discrete and finitely generated and $g$ is any element; or
\item $G$ is a connected semisimple Lie group and $g$ is a semisimple element.
\end{itemize}
Then $\hat A_g(H)$ is independent of the choice of $G$-invariant Riemannian metric, and $\hat A_g(H) = 0$.
\end{corollary}
See Theorem \ref{cor obstr psc}.

Theorem 2 in \cite{HM16} is a generalisation of Atiyah and Hirzebruch's vanishing result \cite{Atiyah70b} to the noncompact case. It states that, if $G$ is connected, 
and not every stabiliser of its action on $H$ is maximal compact, then $\ind_G(D^H) = 0$. This implies that $\hat A_g(H) = 0$ as well. 
In view of Theorem \ref{thm obstr psc intro} and Corollary \ref{cor obstr psc intro}, this makes it a natural question if $M$ admits a $G$-invariant Riemannian metric with positive scalar curvature if $G$ is a connected Lie group. The answer, given in the present paper, turns out to be \emph{yes} under a certain equivariant bounded geometry assumption.

Suppose that $G$ is a connected Lie group, and let $K<G$ be a maximal compact subgroup.
Abels' slice theorem \cite{Abels} implies that there is a diffeomorphism $M \cong G\times_K N$, for a $K$-invariant submanifold $N \subset M$. Consider the infinitesimal action map 
\[
\varphi\colon N \times \kk \to TN
\]
mapping $(y,X) \in N \times \kk$ to $\ddt \exp(tX)y$. The action by $K$ on $N$ is said to have \emph{no shrinking orbits} with respect to a $K$-invariant Riemannian metric on $N$, if the pointwise operator norm of $\varphi$ as a map from $\kk$ to a tangent space is uniformly positive outside a neighbourhood of the fixed point set $N^K$. We say that $N$ has \emph{$K$-bounded geometry} if it has bounded geometry and no shrinking orbits.

\begin{theorem}\label{thm exist intro}
Suppose that  $K$ is non-abelian, and that $K$ acts effectively on $N$ with compact fixed point set. If there exists a $K$-invariant Riemannian metric on $N$ for which $N$ has $K$-bounded geometry, then the manifold \mbox{$G \times_K N$} admits a $G$-invariant metric with uniformly positive scalar curvature.
\end{theorem}

In the compact case, the Atiyah--Hirzebruch vanishing theorem \cite{Atiyah70b} implies that the obstruction $\hat A(N)$ to Riemannian metrics of positive scalar curvature vanishes if $K$ acts nontrivially on $N$. Then Lawson and Yau \cite{Lawson74} constructed such metrics, under mild conditions. In a similar way, Theorem \ref{thm exist intro} complements the vanishing result in \cite{HM16}.


\subsection*{A Callias-type index theorem}

Two effective sources of index-theoretic obstructions to metrics of positive scalar curvature on noncompact manifolds are coarse index theory and Callias-type index theory. For some results involving coarse index theory, see for example \cite{Schick18} and the literature on the coarse Novikov conjecture, in particular
 \cite{FWY19} for the equivariant setting we are interested in here.
We will use Callias-type index theory.

Not assuming that $M$ is $\Spin$ for now, and letting $G$ be any locally compact group, we consider a $G$-equivariant Dirac-type operator $D$ on a $G$-equivariant vector bundle $S \to M$. A \emph{Callias-type operator} is of the form $D + \Phi$, for a $G$-equivariant endomorphism  $\Phi$ of $S$ such that $D + \Phi$ is uniformly positive outside a cocompact set. Then this operator has an index $\ind_G(D + \Phi) \in K_*(C^*(G))$, constructed in \cite{Guo18}. (See Theorem 4.2 in \cite{GHM1} for a realisation of this index in terms of coarse geometry.)
\begin{theorem}[$G$-Callias-type index theorem] \label{thm ind intro}
We have
\[
\ind_G(D + \Phi) = \ind_G(D_N),
\]
for a Dirac operator $D_N$ on a $G$-invariant, cocompact hypersurface $N \subset M$. 
\end{theorem}
See Theorem \ref{thm index}. Versions of this result where $G$ is trivial were proved in \cite{Anghel93, Bott78, Bunke95, Callias78, Kucerovsky01}. Versions for operators on bundles of modules over operator algebras were proved in  \cite{Braverman18, Cecchini16}. Parts of our proof of Theorem \ref{thm ind intro} are based on a similar strategy as the proof of the index theorem in \cite{Cecchini16}.

We deduce Theorem \ref{thm obstr psc} from Theorem \ref{thm ind intro}. This approach is an equivariant generalisation of the obstructions to metrics of positive scalar obtained in \cite{Anghel93, Bunke95, Cecchini16}.

If $g \in G$, then under conditions, there is a subalgebra $A(G) \subset C^*(G)$ such that $K_*(A(G)) = K_*(C^*(G))$, and there is a well-defined trace $\tau_g$ on $A(G)$ given by
\[
\tau_g(f) = \int_{G/Z} f(hgh^{-1})d(hZ),
\]
where $Z$ is the centraliser of $g$. In various settings, including the three cases in Corollary \ref{cor obstr psc intro}, there are index formulas for the number $\tau_g(\ind_G(D))$, see \cite{HW2, Wang14, Wangwang}. These index formulas imply that, in the setting of Theorem \ref{thm obstr psc intro},
\[
\tau_g(\ind_G(D^H)) = \hat A_g(H).
\]
Hence Theorem \ref{thm obstr psc intro} implies Corollary \ref{cor obstr psc intro}.

Theorem \ref{thm exist intro} is proved via a generalisation of Lawson and Yau's arguments \cite{Lawson74}, together with a result from \cite{GHM1} that allows one  to induce up metrics of positive scalar curvature from $N$ to $M = G\times_K N$.

Apart from using Theorem \ref{thm ind intro} to prove Theorem \ref{thm obstr psc}, we obtain some further applications, on the image (Corollary \ref{cor im mu}) and cobordism invariance (Corollary \ref{cor cobord}) of the analytic assembly map; on the Callias-type index of $\Spinc$-Dirac operators  (Corollary \ref{cor Spinc}); on induction of Callias-type indices from compact groups to noncompact groups (Corollary \ref{cor Ind}); and on the $\Spinc$-version \cite{Paradan17} of the \emph{quantisation commutes with reduction} problem \cite{Guillemin82, Meinrenken98, Paradan01, Zhang98} for $\Spinc$-Callias type operators (Corollary \ref{cor quant}).

\subsection*{Outline of this paper}

We state our obstruction and existence results in Section \ref{sec results}: Theorem \ref{thm obstr psc}, Corollary \ref{cor obstr psc} and Theorem \ref{thm:noncptlawsonyauG}.
In Section \ref{sec Callias}, we state the equivariant Callias-type index theorem, Theorem \ref{thm index}. Theorem \ref{thm index} is proved in Sections \ref{sec prop ind} and \ref{sec pf ind}. We then deduce Theorem \ref{thm obstr psc} and Corollary \ref{cor obstr psc} in Subsection \ref{sec pf obstr}. Theorem \ref{thm:noncptlawsonyauG} is proved in Subsection \ref{sec pf exist}. In Section \ref{sec appl}, we obtain some further applications of the Callias-type index theorem, Corollaries \ref{cor im mu}--\ref{cor quant}.

\subsection*{Acknowledgements}

HG was supported in part by funding from the National Science Foundation under grant no.\ 1564398.
PH thanks Guoliang Yu and Texas A{\&}M University for their hospitality during a research visit. 
VM was supported by funding from the Australian Research Council, through the Australian Laureate Fellowship FL170100020.

\subsection*{Notation and conventions}

All manifolds, vector bundles, group actions and other maps between manifolds are implicitly assumed to be smooth. If a Hilbert space $H$ is mentioned, the inner product on that space will be denoted by $(\relbar, \relbar)_H$, and the corresponding norm by $\|\cdot \|_H$. Spaces of continuous sections are denoted by $\Gamma$; spaces of smooth sections by $\Gamma^{\infty}$. Subscripts $c$ denote compact supports.

If $G$ is a group, and $H$ is a subgroup acting on a set $S$, then we write $G \times_H S$ for the quotient of $G \times S$ by the $H$-action given by
\[
h\cdot(g, s) = (gh^{-1}, hs),
\]
for $h \in G$, $g \in G$ and $s \in S$. If $S$ is a manifold, $G$ is  Lie group and $H$ is a closed subgroup, then this action is proper and free, and $G\times_H S$ is a manifold.

A continuous group action, and also the space acted on, is said to be \emph{cocompact} if the quotient space is compact.

\section{Results on positive scalar curvature} \label{sec results}

In all parts of this paper except Subsections \ref{sec exist} and \ref{sec pf exist}, which concern the existence results, $G$ will be be a locally  group, acting properly on a manifold $M$. We do not assume that $M/G$ is compact and are in fact interested mainly in the case where it is not.
The group $G$ may have infinitely many connected components, and for may for example be an infinite discrete group.

\subsection{Obstructions}

For a proper, cocompact action by $G$ on a manifold $N$, a $G$-equivariant elliptic differential operator $D$ on $N$ has an equivariant index $\ind_G(D) \in K_*(C^*(G))$ defined by the analytic assembly map \cite{Connes94}. Here $C^*(G)$ is the maximal or reduced group $C^*$-algebra of $G$, and the index takes values in its even $K$-theory if $D$ is odd with respect to a grading, and in odd $K$-theory otherwise.

Our most general obstruction result is the following.
\begin{theorem}\label{thm obstr psc}
Let $M$ be a complete Riemannian $\Spin$-manifold, on which a locally compact group $G$ acts properly and isometrically. Let $H \subset M$ be a $G$-invariant, cocompact hypersurface with trivial normal bundle,  such that $M \setminus H = X \cup Y$ for disjoint open subsets $X,Y \subset M$. 
If the scalar curvature on $M$ is nonnegative, and positive in a neighbourhood of $H$, then the  $\Spin$-Dirac operator $D^H$ on $H$, acting on sections of the restriction of the spinor bundle on $M$ to $H$, satisfies
\[
\ind_G(D^H) = 0 \quad \in K_*(C^*(G)).
\]
\end{theorem}
We will deduce this result from an equivariant index theorem for Callias-type operators, Theorem \ref{thm index}, which may be of independent interest and has some other applications as well.

\begin{remark} \label{rem Cecchini}
If $M$ is the universal cover of a manifold $X$, and $G$ is the fundamental group of $X$ acting on $M$ in the usual way, then Theorem \ref{thm obstr psc} reduces to Theorem A in \cite{Cecchini16}, by the Mi\v{s}\v{c}enko--Fomenko realisation of the analytic assembly map in that case \cite{Miscenko79}.
\end{remark}

Theorem \ref{thm obstr psc} implies a set of topological obstructions to $G$-invariant positive scalar curvature metrics on $M$. Let $X$ be any Riemannian manifold on which $G$ acts properly, isometrically and cocompactly. Let $g \in G$, and let $X^g \subset X$ be its fixed point set. (Properness of the action implies that $X^g = \emptyset$ if $g$ is not contained in a compact subset of $G$.)
Let $\mathcal{N} \to X^g$ be the normal bundle to $X^g$ in $X$. The connected components of $X^g$ are submanifolds of $X$ of possibly different dimensions, so the rank of $\mathcal{N}$ may jump between these components. In what follows, we implicitly apply all constructions to the connected components of $X^g$ and add the results together.

By a {\it cutoff function} we will mean a smooth function ${\cut}\colon M\rightarrow [0,1]$ such that $\supp({\cut})$ has compact intersection with each $G$-orbit, and  for each $x\in M$ we have
	$$\int_G{\cut}(g^{-1}x)\,dg = 1.$$
We will also use cutoff functions for other group actions, which are defined analogously.

Let $R^{\mathcal{N}}$ be the curvature of the Levi--Civita connection restricted to $\mathcal{N}$. Let $\hat A(X^g)$ be the $\hat A$-class of $X^g$. Let $Z_G(g) < G$ be the centraliser of $g$. Let $c^g$ be a cutoff function for the action by $Z_G(g)$ on $X^g$.
\begin{definition}
The \emph{$g$-localised $\hat A$-genus of $X$} is
\[
\hat A_g(X) \coloneqq \int_{X^g} c^g \frac{\hat A(X^g) }{\det(1-g e^{-R^{\mathcal{N}}/2\pi i})^{1/2}}.
\]
\end{definition}
If $g=e$, then 
\[
\hat A_e(X) = \int_{X} c^e \hat A(X) 
\]
 is the $L^2$-$\hat A$ genus of $X$ used in \cite{Wang14}. If $G$ is discrete and acts properly and freely on $X$, then $\hat A_e(X) = \hat A(X/G)$. 
\begin{corollary} \label{cor obstr psc}
Suppose that either
\begin{itemize}
\item $G$ is any locally compact group and $g = e$;
\item $G$ is discrete and finitely generated and $g$ is any element;
\item $G$ is a connected semisimple Lie group and $g$ is a semisimple element.
\end{itemize}
Let $M$ be a manifold on which $G$ acts properly and that admits a $G$-equivariant $\Spin$-structure. Let $H \subset M$ be a $G$-invariant, cocompact hypersurface such that $M \setminus H = X \cup Y$ for disjoint open subsets $X,Y \subset M$. 
The localised $\hat A$-genus $\hat A_g(H)$ is independent of the choice of a Riemannian metric. If $M$ admits a complete, $G$-invariant Riemannian metric whose 
scalar curvature  is nonnegative, and positive in a neighbourhood of $H$, then $\hat A_g(H) = 0$.
\end{corollary}

Theorem 2 in \cite{HM16} is a generalisation of Atiyah and Hirzebruch's vanishing theorem \cite{Atiyah70b} to actions by noncompact groups. It states that if $G$ is a connected Lie group, and not all stabilisers of the action by $G$ on $H$ are maximal compact, then $\ind_G(D^H)=0$.  So in this setting, the obstructions in Theorem \ref{thm obstr psc} and Corollary \ref{cor obstr psc} vanish. This makes it a natural question whether Riemannian metrics as in Theorem \ref{thm obstr psc} exist if $G$ is a connected Lie group. A partial affirmative answer to that question is given in Subsection \ref{sec exist}. 

This also means that the natural place to look for examples and applications where Theorem \ref{thm obstr psc} and Corollary \ref{cor obstr psc} yield nontrivial obstructions is the setting where $G$ has infinitely many connected components. (The vanishing result  generalises directly to the case where $G$ has finitely many connected components.) As noted in Remark \ref{rem Cecchini}, Theorem \ref{thm obstr psc} implies Theorem A in \cite{Cecchini16}, in the case where $G$ is the fundamental group of $M/G$ and $M$ is its universal cover. More generally, if $G$ is discrete, then $M/G$ is an orbifold. Then Theorem \ref{thm obstr psc} and Corollary \ref{cor obstr psc} lead to obstructions to orbifold metrics on $M/G$ with nonnegative scalar curvature, and positive scalar curvature near the sub-orbifold $N/G$. If $G$ acts freely, then $\hat A_g(H)$ is zero if $g \not=e$ (and $\hat A_e(H) = \hat A(H/G)$), but in the orbifold case the localised $\hat A$-genera for nontrivial elements $g$ are additional obstructions to positive scalar curvature.

\subsection{Existence for connected Lie groups} \label{sec exist}

In this subsection, we suppose that $G$ is a connected Lie group. As pointed out at the end of the previous subsection, for such groups the obstructions to $G$-invariant Riemannian metrics with positive scalar curvature in Theorem \ref{thm obstr psc} and Corollary \ref{cor obstr psc} vanish under a mild assumption on the action, so it is natural to investigate existence of such metrics.

If $G$ is connected, Abels' global slice theorem \cite{Abels} implies we have a diffeomorphism $M \cong G\times_K N$, for a $K$-invariant submanifold $N \subset M$.
Our existence result, Theorem \ref{thm:noncptlawsonyauG}, supposes that such a slice $N$ has
 \emph{$K$-bounded geometry}, a notion introduced in Definition \ref{def K bdd geom}.

   Suppose a compact, connected Lie group $K$ acts isometrically on a complete Riemannian manifold $(N,g_N)$. Let $b$ denote a bi-invariant Riemannian metric on $K$. For each $y\in N$ we have a linear map 
    \begin{equation}
    \label{eq:differentiation}
    \varphi_y\colon \mathfrak{k}\rightarrow T_y N
    \end{equation}
defined by $\varphi_y(X) := \ddt \exp(tX)y$, for $X \in \kk$.    
    Define a pointwise norm function 
    \beq{eq norm phi}
    \|\varphi\|\colon N\rightarrow\mathbb{R},\qquad y\mapsto \|\varphi_y\|,
    \eeq
    where $\|\,\cdot\|$ denotes the linear operator norm with respect to $b$ and $g_N$.
   	\begin{definition}
	\label{def:noshrinkingorbits}
	We say that the action of $K$ on $N$ has \emph{no shrinking orbits} if, for any neighbourhood $U$ of the fixed point set $N^K$, there exists a constant $C_U>0$ such that for all $y\in N\setminus U$ we have
	$$\|{\varphi}(y)\| \geq C_U,$$
	 where the norm function is taken with respect to the Riemannian metric $g$.
	\end{definition}
We remark that the condition of the action having no shrinking orbits is independent of $b$.
\begin{example}
Suppose that $N = \R^2$, on which $K = \SO(2)$ acts in the natural way. Let $\psi \in C^{\infty}(\R^2)$ be positive and rotation-invariant, and consider the Riemannian metric on $\R^2$ equal to $\psi^2$ times the Euclidean metric.
Then for all $y \in \R^2$ and $X \in \R \cong \so(2)$, 
\[
\varphi_y(X) = X \begin{pmatrix} 0 & -1 \\ 1 & 0\end{pmatrix}y.
\]
So 
$
\|\varphi\|(y) = \psi(y)\|y\|,
$
where $\| y \|$ is the Euclidean norm of $y$. Hence the action has no shrinking orbits if and only if the function $y \mapsto \psi(y)\|y\|$ has a positive lower bound outside a neighbourhood of   $(\R^2)^{\SO(2)} = \{0\}$.
%
\end{example}

Now let us define the notion of $K$-bounded geometry, which is a strengthening of the standard notion of bounded geometry.
	\begin{definition}
	A Riemannian manifold has \emph{bounded geometry} if
	\begin{itemize}
	    \item its injectivity radius is positive; 
	    \item for each $l\geq 0$ there exists $C_l>0$ such that $\|{\nabla^l R}\|_\infty\leq C_l$, where $R$ is the Riemann curvature tensor.
	\end{itemize}
	\end{definition}
	\begin{definition} \label{def K bdd geom}
	The action of $K$ on $N$ is said to have \emph{$K$-bounded geometry} if it has no shrinking orbits and $N$ has bounded geometry.
	\end{definition}

Our main existence result, proved in Subsection \ref{sec pf exist}, is the following.	
    	\begin{theorem}
	\label{thm:noncptlawsonyauG}
Let $G$ be a connected Lie group, and $K<G$ a maximal compact subgroup with non-abelian identity component. Let $N$ be a manifold admitting an effective action by $K$ with compact fixed point set. If there exists 
	a
	 Riemannian metric on $N$ such that the $K$-action has $K$-bounded geometry, then the manifold \mbox{$G \times_K N$} admits a $G$-invariant metric with uniformly positive scalar curvature.
	\end{theorem}
This result may be viewed as a strengthening of the vanishing of the obstructions to $G$-invariant metrics of positive scalar curvature in Theorem \ref{thm obstr psc} and Corollary \ref{cor obstr psc} in the case of  connected Lie groups, by the result in \cite{HM16}, in the same way that Lawson and Yau's \cite{Lawson74} construction of metrics of positive scalar curvature strengthens the vanishing of the $\hat A$-genus as in Atiyah and Hirzebruch's vanishing theorem \cite{Atiyah70b} in the compact case. (See the diagram on page 233 of \cite{Lawson74}.)


\begin{remark}
By Abels' slice theorem \cite{Abels}, every manifold with a proper action by a connected Lie group $G$ is of the form $G \times_K N$ in Theorem \ref{thm:noncptlawsonyauG}.
The condition that $N^K$ is compact is equivalent to the condition that the points in $G \times_K N$ whose stabilisers in $G$ are maximal compact form a cocompact set.
\end{remark}

\section{An index theorem} \label{sec Callias}

We will deduce Theorem \ref{thm obstr psc}, and hence Corollary \ref{cor obstr psc}, from an equivariant index theorem for Callias-type operators, Theorem \ref{thm index}. This is based on equivariant index theory for such operators with respect to proper actions, developed in \cite{Guo18}. The proof of the index theorem involves several arguments analogous to those in \cite{Cecchini16}.

\subsection{The $G$-Callias-type index} \label{sec def index}

%

From now on, $M$ will be a complete Riemannian manifold on which $G$ acts properly and isometrically. Let $S \to M$ denote a $\Z/2$-graded, $G$-equivariant Clifford module over $M$, and $D$ an odd-graded Dirac operator on $\Gamma^{\infty}(S)$, associated to a $G$-invariant Clifford connection on $S$ via the Clifford action by $TM$ on $S$. 

Let $\Phi$ be an odd, $G$-equivariant, fibrewise Hermitian vector bundle endomorphism of $S$. 
\begin{definition} \label{def Phi}
The endomorphism $\Phi$ is \emph{admissible} for $D$ if
\begin{itemize}
\item the operator $D\Phi + \Phi D$ on $\Gamma^{\infty}(S)$ is a vector bundle endomorphism; and
\item there are a cocompact subset $Z \subset M$ and a constant $C>0$ such that we have the pointwise estimate
\beq{eq est Phi}
 \Phi^2 \geq \|D\Phi + \Phi D\| + C
\eeq
on $M \setminus Z$.
\end{itemize}
In this setting the operator $D + \Phi$ is called a \emph{$G$-Callias-type operator}.
\end{definition}

In the rest of the paper, we will use the following notation. Let $\mu$ denote the modular function on $G$. 
Let $C^*(G)$ be either the reduced or maximal group $C^*$-algebra of $G$. Let $G$ act on sections of the bundle $S$ by
$$(gs)(x)\coloneqq g(s(g^{-1}x)),$$
for $s$ a section,  $g\in G$ and $x \in M$. 

Equip the space $\Gamma_c(S)$ with a right $C_c(G)$-action defined by
\beq{eq s1b}
(s_1\cdot b)(x) \coloneqq  \int_G (gs)(x)\cdot b(g^{-1})\mu(g)^{-1/2}\,dg
\eeq
and a $C_c(G)$-valued inner product defined by
\beq{eq inner prod CstarG}
(s_1, s_2)_{C^*(G)}(g) \coloneqq  \mu(g)^{-1/2}(s_1, gs_2)_{L^2(S)},
\eeq
for $s_1, s_2 \in \Gamma^\infty_c(S)$,  $b\in C_c(G)$, $g \in G$, and $x \in M$. Let $\cE$ be the Hilbert $C^*(G)$-module completion of $\Gamma^\infty_c(S)$ with respect to this structure.

The definition of the equivariant index of $G$-Callias-type operators is based on the following result, Theorem 4.19 in \cite{Guo18}.
\begin{theorem} \label{thm Guo}
There is a continuous $G$-invariant cocompactly supported function $f$ on $M$ such that 
\beq{eq def F}
F\coloneqq  (D+ \Phi)\bigl(  (D + \Phi)^2 + f \bigr)^{-1/2},
\eeq
is a well-defined, adjointable operator on $\cE$, such that $(\cE, F)$ is a Kasparov $(\C, C^*(G))$-cycle. Its class in $\KK(\C, C^*(G))$ is independent of the function $f$ chosen.
\end{theorem}
For details about the definition of the operator $F$, we refer to Definition 4.11 in \cite{Guo18}.

\begin{definition} \label{def index}
The \emph{$G$-index} of the $G$-Callias-type operator $D + \Phi$ is the class
\[
\ind_G(D + \Phi) \coloneqq  [\cE, F] \in K_0(C^*(G)) = \KK(\C, C^*(G))
\]
as in Theorem \ref{thm Guo}.
\end{definition}

\subsection{Hypersurfaces and the index theorem} \label{sec index thm}

In the setting of the previous subsection, we now suppose that $S = S_0 \oplus S_0$ for an ungraded, $G$-equivariant Clifford module $S_0$ over $M$, where the first copy of $S_0$ is the even part of $S$, and the second copy is the odd part. Suppose that 
\beq{eq D D0}
D = 
\begin{pmatrix}
0 & D_0 \\ D_0  & 0
\end{pmatrix}
\eeq
for a Dirac operator $D_0$ on $S_0$, and that
\beq{eq Phi Phi0}
\Phi = \begin{pmatrix}
0 & i\Phi_0 \\ -i \Phi_0  & 0
\end{pmatrix}
\eeq
for a Hermitian endomorphism $\Phi_0$ of $S_0$. (The conditions on  $D\Phi + \Phi D$ then become conditions on $[D_0, \Phi_0]$.)

Let $Z$ be as in Definition \ref{def Phi}. Let $M_- \subset M$ be a $G$-invariant, cocompact subset containing $Z$ in its interior, such that $N \coloneqq  \partial M_-$ is a smooth submanifold of $M$. Let $M_+$ be the closure of the complement of $M_-$, so that $N = M_- \cap M_+$ and $M = M_- \cup M_+$.  In this and similar settings, we write
\[
M = M_- \cup_N M_+.
\]

By \eqref{eq est Phi}, the restriction of $\Phi_0$ to $N$ is fibrewise invertible. Let $S^N_+ \to N$ and $S^N_- \to N$ be its positive and negative eigenbundles. 
(These are vector bundles, even though  eigenbundles for single eigenvalues may not be.)
Clifford multiplication by the unit normal vector field $\hat n$ to $N$ pointing into $M_+$, times $-i$,  defines  $G$-invariant gradings on $S^N_{\pm}$.

Let $\nabla^{S_0}$ be the Clifford connection on $S_0$ used to define $D_0$. By restriction and projection, it defines connections $\nabla^{S^N_{\pm}}$ on $S^N_{\pm}$. The Clifford action by $TM|_N$ on $S_0|_N$ preserves $S^N_{\pm}$ by the first condition in Definition \ref{def Phi}; see also Remark 1.2 in \cite{Anghel93}. 
Hence the connections $\nabla^{S^N_{\pm}}$ define Dirac operators $D^{S^N_{\pm}}$ on $\Gamma^{\infty}(S^N_{\pm})$. These operators are  odd-graded. Because $N$ is cocompact, $D^{S^N_+}$ has an equivariant index
\[
\ind_G(D^{S^N_+}) \in K_0(C^*(G))
\]
defined by the analytic assembly map \cite{Connes94}.
\begin{theorem}[$G$-Callias-type index theorem]\label{thm index}
We have
\beq{eq index}
\ind_G(D + \Phi) =\ind_G(D^{S^N_+}) \quad \in K_0(C^*(G)).
\eeq
\end{theorem}
Versions of this result where $G$ is trivial were proved in \cite{Anghel93, Bott78, Bunke95, Callias78, Kucerovsky01}. Versions for operators on bundles of modules over operator algebras are proved in  \cite{Braverman18, Cecchini16}.

There are various index theorems for the the image of the right hand side of \eqref{eq index} under traces  \cite{HW2, Wang14, Wangwang} or pairings with higher cyclic cocycles \cite{HST19, Pflaum15, Piazza18}. Via these results, Theorem \ref{thm index} yields topological expressions for the corresponding images of the left hand side of \eqref{eq index}. The results in \cite{HW2, Wang14, Wangwang} will be used to deduce Corollary \ref{cor obstr psc} from Theorem \ref{thm obstr psc}.

\section{Properties of the $G$-Callias-type index} \label{sec prop ind}

To prove Theorem \ref{thm index}, we will make use of the properties of the index of Definition \ref{def index} that we describe below.

\subsection{Sobolev modules}

We start by recalling the definition of Sobolev Hilbert $C^*(G)$-modules from \cite{Guo18}. Let $M$, $G$,  $S$ and $D$ be as in Subsection \ref{sec def index}.
\begin{definition}
		\label{def sobolev}
		For each nonnegative integer $j$, define $\Gamma_c^{\infty,j}(S)$ to be the pre-Hilbert $C_c(G)$-module whose underlying vector space is $\Gamma_c^{\infty}(S)$, equipped with the right $C_c(G)$-action defined by \eqref{eq s1b}, and $C_c(G)$-valued inner product defined by
\[
\langle e_1,e_2\rangle_{\cE^j}=\sum_{k=0}^j (D^k e_1, D^k e_2)_{C^*(G)},
\]
where $e_1,e_2\in \Gamma_c^\infty(S)$ and $(\relbar, \relbar)_{C^*(G)}$ is as in \eqref{eq inner prod CstarG}. Here we set $D^0$ equal to the identity operator. Denote by $\mathcal{E}^j(S)=\mathcal{E}^j$ the vector space completion of $\Gamma_c^{\infty,j}(S)$ with respect to the norm 
induced by $\langle\relbar,\relbar\rangle_{\cE^j}$, and extend naturally the $C_c^\infty(G)$-action to a $C^*(G)$-action, and $\langle\relbar,\relbar\rangle_{\cE^j}$ to a $C^*(G)$-valued inner product on $\mathcal{E}^j$, to give it the structure of a Hilbert $C^*(G)$-module. 
We call $\mathcal{E}^j$ the {\it $j$-th $G$-Sobolev module with respect to $D$}. 
\end{definition}

The module $\cE$ defined above Theorem \ref{thm Guo} equals $\cE^0$.
The following version of the Rellich lemma holds for Sobolev modules (Theorem 3.12 in \cite{Guo18}).
\begin{theorem}
		\label{thm Rellich}
		Let $f$ be a continuous $G$-invariant cocompactly supported function on $M$. Then multiplication by $f$ defines an element of $\mathcal{K}(\mathcal{E}^{s},\mathcal{E}^{t})$ whenever $s>t$.
\end{theorem}

We will state and prove a homotopy invariance result, Proposition \ref{prop htp invar}, for the index in Definition \ref{def index}, that will be of use later. A hypothesis in this result is that a certain vector bundle endomorphism defines adjointable operators on the Sobolev modules $\cE^0$ and $\cE^1$. In order to check this condition in some geometric situations relevant to us, we will need Propositions \ref{prop E0} and \ref{prop:boundedhilbert} below.

\begin{proposition}
	\label{prop E0}
%
 A smooth, $G$-invariant, uniformly bounded bundle endomorphism of $S$
 defines an element of $\mathcal{L}(\mathcal{E}^0)$.
\end{proposition}
\begin{proof}
Let $\Psi$ be a smooth, $G$-invariant, uniformly bounded bundle endomorphism of $S$.
	Since $\Psi$ is uniformly bounded, it defines a bounded operator on $L^2(S)$. Let $\norm{\Psi}$ denote its operator norm, let ${\cut}$ be a cutoff function on $M$, and let $\Psi^*$ be the pointwise adjoint of $\Psi$.
	Since the operator $\Psi_1\coloneqq \Psi^*\Psi-\norm{\Psi}^2$ is positive on $L^2(S)$, it has a positive square root $Q$ that one observes is $G$-invariant. For a fixed $e \in \Gamma^{\infty}_c(S)$, the function
\[
g\mapsto ({\cut}\Psi_1(ge),ge)_{L^2(S)}=(\sqrt{{\cut}}Q(ge),\sqrt{{\cut}}Q(ge))_{L^2(S)}
\]
	has compact support in $G$, by properness of the $G$-action. Thus the map $G\to L^2(S)$ defined by $g\mapsto\sqrt{{\cut}}Q(ge)$ has compact support in $G$.


It follows that for any unitary representation $\pi$ of $G$ on a Hilbert space $H$ and $h\in H$,
\[	
	v\coloneqq\int_G\mu(g)^{-1/2} \sqrt{{\cut}} Q(ge)\otimes\pi(g)h\,dg
\]
	is a well-defined vector in $L^2(S)\otimes H$. By computations similar to those in the proof of Proposition 5.4 in \cite{Guo18}, one sees that $\norm{v}_{L^2(S) \otimes H}$ equals
	\[
	\int_G \int_G\left \langle g c \Psi_1 g^{-1}e,e\right\rangle_{\mathcal{E}^0} (g') \,dg \cdot(\pi(g')h,h)_H\,dg'.
\]	
	Thus, for any unitary representation $\pi$ of $G$, 
\[	
	\pi \left(
	 \int_G\left \langle g c \Psi_1 g^{-1}e,e\right\rangle_{\mathcal{E}^0}  \,dg 
	\right) = \pi ( \langle \Psi_1 e ,e\rangle_{\mathcal{E}^0} )
\]	
	is a positive operator, where we let $f\in C_c(G)$ act on $H$ by 
\[	
\pi(f) h\coloneqq 	 \int_G f(g)\pi(g)h\,dg.
\]		
	It follows that the element
\[
\langle\Psi_1 e,e\rangle_{\mathcal{E}^0}=\left\langle\left(\Psi^*\Psi-\norm{\Psi}^2\right)e,e\right\rangle_{\mathcal{E}^0}=\langle\Psi e,\Psi e\rangle_{\mathcal{E}^0}-\norm{\Psi}^2\langle e,e\rangle_{\mathcal{E}^0}.
\]
	is positive in $C^*(G)$. 
	Hence $\Psi$ extends to an operator on all of $\mathcal{E}^0$. Similarly, $\Psi^*$ defines an operator on all of $\mathcal{E}^0$ that one checks is the adjoint of $\Psi$.
\end{proof}

\begin{proposition}
		\label{prop:boundedhilbert}
		Suppose that there are a $G$-invariant, cocompact subset $K \subset M$ and a $G$-invariant, cocompact hypersurface $N \subset M$ such that there is a $G$-equivariant isometry $M \setminus K \cong N \times (0,\infty)$, and a $G$-equivariant vector bundle isomorphism $S|_{M \setminus K} \cong S|_N \times (0,\infty)$.
Let $\Psi$ be a $G$-equivariant vector bundle endomorphism of $S$. Suppose that, on $M\setminus K$,  $\Psi$ and $D$ are constant in the factor $(0,\infty)$ of $M \setminus K \cong N \times (0,\infty)$.
Then $\Psi$ defines an element of $\mathcal{L}(\mathcal{E}^1)$.
\end{proposition}

The proof uses the next lemma. To state it, let $H^1(S)$ be the completion of $\Gamma_c^\infty(S)$ with respect to the inner product 
$$(\relbar, \relbar)_{H^1(S)}=(\relbar,\relbar )_{L^2(S)}+ ( D \relbar,D \relbar)_{L^2(S)}.$$
\begin{lemma}
	\label{lem positivity}
	Let $M$ and $S$ be as in Proposition \ref{prop:boundedhilbert}. Let $\Theta$ be a bounded, positive operator on $H^1(S)$ such that
	\begin{itemize}
	\item $\Theta$ preserves the subspace $\Gamma_c^\infty(S)$;		
 	\item for any $e\in \Gamma_c^\infty(S)$, the function $a\colon G\rightarrow\mathbb{\mathbb{R}}$ given by $g\mapsto (\Theta(ge),ge)_{H^1(S)}$ has compact support in $G$.
 	\end{itemize}	
 	Then 
	\[
	\int_G
	\left\langle g\Theta g^{-1}  e,e\right\rangle_{\mathcal{E}^1}\, dg
	\]
	is a positive element of $C^*(G)$.
\end{lemma}
\begin{proof}
	Let $Q$ be the positive square root of $\Theta$ in $\cB(H^1(S))$.
		Since $a$ has compact support, and $(\Theta(ge),ge)_{H^1(S)}= ( Q(ge),Q(ge))_{H^1(S)}$,
	the map $G\to H^1(S)$ defined by $g\mapsto Q(ge)$ has compact support in $G$.
As in the proof of Proposition \ref{prop E0}, one finds that for any unitary representation $\pi$ of $G$ on a Hilbert space $H$ and $h\in H$,
\[	
	v\coloneqq\int_G\mu(g)^{-1/2}Q(ge)\otimes\pi(g)h\,dg
\]
	is a well-defined vector in $H^1(S)\otimes H$, and that 
\[
\int_G \int_G \left\langle gQ^2 g^{-1}e,e\right\rangle_{\mathcal{E}^1}(g') \,dg\cdot(\pi(g')h,h)_H\,dg'
=
\|v\|_{H^1(S) \otimes H} \geq 0.
\]	
Similarly to the proof of Proposition \ref{prop E0}, we deduce that
$\int_G \langle g \Theta g^{-1}e,e\rangle_{\mathcal{E}^1} \,dg$ is a positive element of $C^*(G)$.	
\end{proof}

	\begin{proof}[Proof of Proposition~\ref{prop:boundedhilbert}]
Because of the forms of $M$ and $S$, there is a canonical (up to equivalence) first Sobolev norm $\|\cdot \|_{1}$ on sections of $S$ that is $G$-invariant, and invariant under the relevant class of translations in the factor $(0, \infty)$ of $N \times (0,\infty)$.  Because $\Psi$ is an order zero differential operator constant on the factor $(0, \infty)$, it defines a bounded operator with respect to $\|\cdot \|_{1}$. Due to the form of $D$, the norm on $H^1(S)$ is equivalent to $\|\cdot \|_{1}$, and so $\Psi$ defines a bounded operator on $H^1(S)$.	
	
	Let $\pi_N\colon M\setminus K\cong N\times(0,\infty)\rightarrow N$ be the natural projection. Let ${\cut}$ be a cutoff function on $M$ such that 
	\[
	{\cut}|_{M \setminus K} = \pi_N^*{\cut}_N
	\]
for a cutoff function ${\cut}_N$ on $N$.	
 Let $\Psi^*$ and ${\cut}^*$ denote the respective adjoints of $\Psi$ and ${\cut}$ in $\cB(H^1(S))$.  Then the operator 
\[
\Psi_1\coloneqq\frac{{\cut}\Psi^*\Psi+\Psi^*\Psi{\cut}^*}{2}
\]
is bounded and self-adjoint on $H^1(S)$ with norm at most $\norm{\Psi}^2\norm{{\cut}}$, where the norms are taken in $\cB(H^1(S))$. 

Let ${\cut}'$ be a smooth, nonnegative function on $M$   that is identically $1$ on the support of $c$, and such that
	\[
	{\cut}'|_{M \setminus K} = \pi_N^*{\cut}'_N
	\]
for a compactly supported function ${\cut}'_N$ on $N$.	Consider the endomorphism 
$\Psi_2\coloneqq({\cut}')^*{\cut}'\norm{\Psi}^2\norm{{\cut}}-\Psi_1$ of $S$. For the same reasons as for $\Psi$, it defines a bounded operator on $H^1(S)$. 
Fix $e \in \Gamma_c^{\infty}(S)$. Because $\Psi$ is a positive bounded operator on $H^1(S)$, 
 we may apply  Lemma \ref{lem positivity} with $\Theta = \Psi_2$ to conclude that
\beq{eq Psi2 nonneg}
\int_G
\left\langle \left( g \Psi_2 g^{-1}\right) e,e\right\rangle_{\mathcal{E}^1} \,dg
=\int_G \left\langle\left(g ( ({\cut}')^*{\cut}' ) g^{-1}\norm{\Psi}^2\norm{{\cut}}\right)e,e\right\rangle_{\mathcal{E}^1}\,dg - \langle \Psi^*\Psi e, e\rangle_{\mathcal{E}^1}
\eeq
		is positive in $C^*(G)$. Define  $b\colon M\times G\rightarrow\mathbb{R}$ by
\[
b(x,g)\coloneqq   b_x(g) \coloneqq  {\cut}'(g^{-1}x).
\]
		By construction of ${\cut}'$, the quantities
\[		
C_1(x)\coloneqq \int_G {\cut}'(g^{-1}x)^2\,dg,\qquad C_2(x)\coloneqq \norm{d{\cut}'}^2_\infty\cdot\vol\left(\supp_G(b_x)\right)
\]
are bounded as functions of $x \in M$		
 (see also Remark \ref{rem:constants} below). A direct calculation shows that
		\begin{align*}
		\left\| \int_G \left\langle\left( g(({\cut}')^*{\cut}') g^{-1}\norm{\Psi}^2\norm{{\cut}}\right)e,e\right\rangle_{\mathcal{E}^1} \,dg\right\|_{C^*(G)}&=\norm{\Psi}^2\norm{{\cut}}\left\|\int_G\left\langle{\cut}'g^{-1}e,{\cut}'g^{-1}e\right\rangle_{\mathcal{E}^1}dg\right\|_{C^*(G)}\\
		&\leq\norm{\Psi}^2\norm{{\cut}}\left( \|C_1\|_{\infty} \norm{e}^2_{\mathcal{E}^1}+ \|C_2\|_{\infty} \norm{e}^2_{\mathcal{E}^0}\right)\\
		&\leq C_3\norm{\Psi}^2\norm{e}^2_{\mathcal{E}^1},
		\end{align*}
		for some constant $C_3$. 
Together with positivity of \eqref{eq Psi2 nonneg}, this implies that
		$$\norm{\Psi e}_{\mathcal{E}^1}^2=\norm{\langle e,\Psi^*\Psi e\rangle_{\mathcal{E}^1}}_{C^*(G)}\leq C_3\norm{\Psi}^2\norm{e}^2_{\mathcal{E}^1},$$
	 	so that $\Psi$ extends to an operator on all of $\mathcal{E}^1$. Similarly, the $H^1(S)$-adjoint $\Psi^*$ defines an operator on $\mathcal{E}^1$ that one checks is the adjoint of $\Psi$.
	\end{proof}
\begin{remark}
\label{rem:constants}
As can be seen from the proof, the conclusion of Proposition \ref{prop:boundedhilbert} holds more broadly for any $M$ on which the functions $C_1$ and $C_2$ on $M$ are  bounded.
\end{remark}

\subsection{Vanishing}

Two cases where the index of Definition \ref{def index} vanishes are straightforward to prove, but we state them here because they will be used in various places.

\begin{lemma}\label{lem invtble}
If \eqref{eq est Phi} holds on all of $M$, then $\ind_G(D + \Phi) = 0$.
\end{lemma}
\begin{proof}
In this setting, the operator $F$ in \eqref{eq def F} is invertible. This implies that the $\KK$-cycle $(\cE, F)$
 is operator homotopic to the degenerate cycle $(\cE, F(F^*F)^{-1/2})$. 
%
\end{proof}

\begin{lemma}\label{lem cocpt}
If $M/G$ is compact, and $D$ and $\Phi$ are of the forms \eqref{eq D D0} and \eqref{eq Phi Phi0}, then $\ind_G(D)=0$.
\end{lemma}
\begin{proof}
In this setting, $\Phi$ is bounded, and the cycle $(\cE, F)$ is operator homotopic to $\Bigl( \cE, \frac{D}{\sqrt{D^2+1}}\Bigr)$. 

In general, let $A$ be a $C^*$-algebra, let  $\cE_0$ be a Hilbert $A$-module, and set $\cE \coloneqq  \cE_0 \oplus \cE_0$.
and an adjointable operator $F$ on $\cE$ such that $(\cE, F)$ is a Kasparov $(\C, A)$-cycle, and $F$ is of the form
\beq{eq F F0}
F = \begin{pmatrix}
0 & F_0 \\ F_0 & 0
\end{pmatrix},
\eeq
for a (necessarily self-adjoint) $F_0 \in \calL(\cE_0)$. Then $[\cE, F] = 0 \in K_0(A)$. Because $\frac{D}{\sqrt{D^2+1}}$ is of the form \eqref{eq F F0}, the claim follows. 
\end{proof}

\subsection{Homotopy invariance}

The index of Definition \ref{def index} has a homotopy invariance property analogous to Proposition 4.1 in \cite{Cecchini16}. This homotopy invariance applies in a more general setting than Callias-type operators.

Let $P$ be an odd, $G$-equivariant Dirac-type operator on a $\Z/2$-graded Clifford module $\cS$, and let $\Psi$ be an odd, smooth $G$-equivariant, uniformly bounded Hermitian vector bundle endomorphism of $\cS$. Fix $t_0 < t_1 \in \R$. For $t \in [t_0, t_1]$, consider the operator $P_t \coloneqq  P+t\Psi$.
\begin{proposition}[Homotopy invariance] \label{prop htp invar}
Suppose that 
\begin{enumerate}
\item for $j=0,1$, the endomorphism $\Psi$ defines an adjointable operator on the Sobolev module $\cE^j$ of Definition \ref{def sobolev};
\item there is a nonnegative, $G$-invariant, cocompactly supported function $f \in C^{\infty}(M)$ such that for all $t \in [t_0, t_1]$, the operator $P_t^2+f\colon \cE^2 \to \cE^0$ is invertible, with inverse in $\calL(\cE^0,  \cE^2)$.
\end{enumerate}
Then $\ind_G(P_t) \in K_0(C^*(G))$ is independent of $t \in [t_0, t_1]$.
\end{proposition}
\begin{proof}
The proof proceeds identically to the proof of Proposition 4.1 in \cite{Cecchini16}, with the exceptions that Theorem \ref{thm Rellich} should be substituted for Lemma 4.2 in \cite{Cecchini16}, and Lemmas 4.4 and 4.6(a) in \cite{Guo18} should be substituted for Lemmas 1.4 and 1.5 in \cite{Bunke95}, respectively.
\end{proof}

\begin{corollary}\label{cor cocpt change}
If $D + \Phi$ is a $G$-Callias-type operator on $S \to M$, and $\Psi$ is a $G$-equivariant, odd vector bundle endomorphism of $S$ that equals zero outside a cocompact set, then $\ind_G(D+\Phi)=\ind_G(D+\Phi+\Psi)$.
\end{corollary}
\begin{proof}
We set $P = D + \Phi$ and apply Proposition \ref{prop htp invar}. Since $\Psi$ is cocompactly supported, the first condition in Proposition \ref{prop htp invar} holds by Proposition 3.5 in \cite{Guo18}. 
 For the same reason, $\Phi+t\Psi$ is a Callias-type potential for all $t \in \R$, so by Theorem 5.6 in \cite{Guo18}, the second condition in Proposition \ref{prop htp invar} holds for $t \in [0,1]$, where a priori the function $f$ may depend on $t$. But since $\Psi$ is zero outside a cocompact set, we can choose $f$ independent of $t$. The claim then follows from Proposition \ref{prop htp invar}.
\end{proof}


\begin{remark}
In Proposition \ref{prop htp invar}, it is not assumed that $\Psi$ is a Callias-type potential in the sense of Definition \ref{def Phi}. We will use Proposition \ref{prop htp invar} in this greater generality in the proof of Lemma \ref{lem DC tilde}.
\end{remark}

\begin{remark}
Corollary \ref{cor cocpt change} can be used to give an alternative proof of Lemma \ref{lem invtble}: this corollary implies that in the setting of that lemma,
 \[
 \ind_G(D + \Phi) =  \ind_G(D - \Phi) =  \ind_G\bigl( (D + \Phi)^* \bigr) = - \ind_G(D + \Phi).
 \]
 See also Corollary 4.9 in \cite{Cecchini16}.
\end{remark}

\subsection{A relative index theorem}

We will use an analogue of Bunke's relative index theorem, Theorem 1.2 in \cite{Bunke95}. For $j=1,2$, let $M_j$, $S_j$, $D_j$ and $\Phi_j$, respectively, be as $M$, $S$, $D$ and $\Phi$ were before.
Suppose that there are a $G$-invariant hypersurface $N_j \subset M_j$ and a $G$-invariant tubular neighbourhood $U_j$ of $N_j$, and that there is a $G$-equivariant isometry $\psi\colon U_1 \to U_2$ such that
\begin{itemize}
\item 
 $\psi(N_1) = N_2$;
 \item  $\psi^*(S_2|_{U_2}) \cong S_1|_{U_1}$;
 \item $\psi^*(\nabla^{S_2}|_{U_2}) = \nabla^{S_1}|_{U_1}$, where $\nabla^{S_j}$ is the Clifford connection used to define $D_j$;  and
 \item  $\Phi_1|_{U_1}$ corresponds to $\Phi_2|_{U_2}$ via $\psi$. 
\end{itemize}

Suppose that $M_j = X_j \cup_{N_j} Y_j$ for closed, $G$-invariant subsets $X_j, Y_j \subset M_j$. We identify $N_1$ and $N_2$ via $\psi$ and write $N$ for this manifold when we do not want to distinguish between $N_1$ and $N_2$.
Using the map $\varphi$, form
\[
M_3 \coloneqq  X_1  \cup_N Y_2; \qquad M_4 \coloneqq  X_2 \cup_N Y_1.
\]
For $j=3,4$, let $S_j$, $D_j$ and $\Phi_j$  be obtained from the corresponding data on $M_1$ and $M_2$ by cutting and gluing along $U_1 \cong U_2$ via $\psi$.
\begin{theorem} \label{thm rel index}
In the above situation,
\[
\ind_G(D_1 + \Phi_1) + \ind_G(D_2 + \Phi_2) = \ind_G(D_3 + \Phi_3) +  \ind_G(D_4 + \Phi_4). 
\]
\end{theorem}
\begin{proof}
This proof is an adaptation of the proof of Theorem 1.14 in \cite{Bunke95}, with some results from \cite{Guo18} added. For $j=1,2,3,4$, let $\cE_j$ and $F_j$ be as $\cE$ and $F$ above and in Theorem \ref{thm Guo}, for the data indexed by $j$. Using superscripts $\op$ to denote opposite gradings, we write
$\cE \coloneqq  \cE_1 \oplus \cE_2 \oplus \cE_3^{\op} \oplus \cE_4^{\op}$ and $F \coloneqq  F_1 \oplus F_2 \oplus F_3 \oplus F_4$. We will show that
\beq{eq 1234}
[\cE, F] = 0 \quad \in \KK_0(\C, C^*(G)),
\eeq
which is equivalent to the theorem.

For $j=1,2$, let $\chi_{X_j}, \chi_{Y_j} \in C^{\infty}(M_j)$ be real-valued functions such that 
\begin{itemize}
\item $\supp(\chi_{X_j}) \subset X_j \cup U_j$ and $\supp(\chi_{Y_j}) \subset Y_j \cup U_j$;
\item $\psi^*(\chi_{X_2}|_{U_2}) = \chi_{X_1}|_{U_1}$ and $\psi^*(\chi_{Y_2}|_{U_2}) = \chi_{Y_1}|_{U_1}$; and
\item $\chi_{X_j}^2 + \chi_{Y_j}^2 = 1$.
\end{itemize}
We view pointwise multiplication by these functions as operators
\beq{eq chi XYj}
\begin{split}
\chi_{X_1}\colon & \cE_1 \to \cE_3;\\
\chi_{X_2}\colon & \cE_1 \to \cE_4;
\end{split}
\qquad
\begin{split}
\chi_{Y_1}\colon & \cE_2 \to \cE_3;\\
\chi_{Y_2}\colon & \cE_2 \to \cE_4.
\end{split}
\eeq
The adjoints of these operators map in the opposite directions, and are also given by pointwise multiplication by the respective functions.
Using these multiplication operators, and the grading operator $\gamma$,
we form the operator
\[
X \coloneqq  \gamma \begin{pmatrix}
0 & 0 & -\chi_{X_1}^* & \chi_{X_2}^* \\
0 & 0 & - \chi_{Y_1}^* & \chi_{Y_2}^* \\
\chi_{X_1} & \chi_{Y_1} & 0 & 0 \\
\chi_{X_2} & -\chi_{Y_2} & 0 & 0
\end{pmatrix}
\]
on $\cE$. Then $X$ is an odd, self-adjoint, adjointable operator on $\cE$ such that $X^2 = 1$. As such, it generates a Clifford algebra $\Cl$.

We claim that $XF + FX$ is a compact operator. This is based on the Rellich lemma for Hilbert $C^*(G)$-modules, Theorem \ref{thm Rellich}. Let $\chi$ be one of the functions $\chi_{X_j}$ or $\chi_{Y_j}$, viewed as an operator from $\cE_k$ to $\cE_l$ as in \eqref{eq chi XYj}. 
Let $f_j \in C^{\infty}(M)$ be as in Theorem \ref{thm Guo}, for the operator $D_j + \Phi_j$. For $j=1,2,3,4$ and $\lambda \in \R$, the operator $(D_j^2 + f_j^2 + \lambda^2)$ on $\cE_j$ is invertible by Lemma 4.6 in \cite{Guo18}, and we denote its inverse by $R_j(\lambda)$.

Then, as in the proof of Theorem 1.14 in \cite{Bunke95}, and using Proposition 4.12 in \cite{Guo18}, we find that
 the operator
\beq{eq chi F}
\chi^* \circ  F_l - F_k \circ  \chi^*\colon \cE_l \to \cE_k 
\eeq
equals
\begin{multline} \label{eq 1234 int}
\frac{2}{\pi} 
\int_{\R}
\Bigl(
-\grad(\chi)R_l(\lambda) + D_k^2 R_k(\lambda) \grad(\chi) R_l(\lambda) +
\Bigr.
\\
\Bigl.
 D_k R_k(\lambda) \grad(f_k) R_k(\lambda) \grad(\chi) R_l(\lambda) + D_k R_k(\lambda) \grad(\chi) D_l R_l(\lambda)
 \Bigr)\, d\lambda.
\end{multline}
Theorem \ref{thm Rellich}, together with Lemma 4.6(a)
in \cite{Guo18}, implies that for all cocompactly supported continuous functions $\varphi$ on $M_j$, the compositions $\varphi D_j^nR_j(\lambda)$, $D_j^n\varphi R_j(\lambda)$ and $D_j^nR_j(\lambda) \varphi$ are compact operators on $\cE_j$ if $n=0,1$, and  adjointable operators if $n=1$. So all the terms in the integrand in \eqref{eq 1234 int} are compact operators. By  Lemmas 4.6 and 4.8 in \cite{Guo18}, the norm of the integrand in \eqref{eq 1234 int} is bounded  by $a(b+\lambda^2)^{-1}$ for constants $a,b>0$. So
the integral converges in the operator norm on $\calL(\cE)$, and we conclude that \eqref{eq chi F} is a compact operator on $\cE$. This implies that $XF + FX$ is a compact operator. 

Because $X$ generates $\Cl$ and $X$ anticommutes with $F$ modulo compacts, the pair $(\cE, F)$ is a Kasparov $(\Cl, C^*(G))$-cycle. Its class in $\KK(\Cl, C^*(G))$ 
is mapped to the left-hand aside of \eqref{eq 1234} by  the pullback along the inclusion map $\C \hookrightarrow \Cl$.
 That map is zero by Lemma 1.15 in \cite{Bunke95}, so \eqref{eq 1234} follows.
\end{proof}

Theorem \ref{thm rel index} implies the following version of Proposition 5.9 in \cite{Cecchini16}. 
\begin{corollary} \label{cor 1234}
In the setting of Theorem \ref{thm rel index}, suppose that for $j= 1,2$, the set $X_j$ is cocompact, and contains a set $Z_j$ for $\Phi_j$ as in Definition \ref{def Phi}. Then
\[
\ind_G(D_1 + \Phi_1) = \ind_G(D_2+ \Phi_2).
\]
\end{corollary}
\begin{proof}
This fact can be deduced from Theorem \ref{thm rel index} in exactly the same way Proposition 5.9 in \cite{Cecchini16} is deduced from Theorem 5.7 in \cite{Cecchini16}. Compared to that proof in \cite{Cecchini16}, references to Corollaries 3.4 and 4.9 and Theorem 5.7 in that paper should be replaced by references to  Lemmas \ref{lem invtble} and \ref{lem cocpt} and Theorem \ref{thm rel index}, respectively, in the present paper.
\end{proof}
The crucial assumption in Corollary \ref{cor 1234} is that all data near $N_1$ can be identified with the corresponding data near $N_2$.

\section{Proof of the $G$-Callias-type index theorem} \label{sec pf ind}

The first and most important step in the proof of Theorem \ref{thm index} is Proposition \ref{prop cylinder}, which states that $\ind_G(D + \Phi)$ equals the index of a $G$-Callias-type operator on the manifold $N \times \R$, which we will call the cylinder on $N$. See Figure \ref{fig cylinder}. Such an approach is taken in proofs of various other index theorems for Callias-type operators; see for example \cite{Anghel93, Braverman18, Bunke95, Cecchini16}. \begin{figure}
\begin{center}
\includegraphics[width=0.7\textwidth]{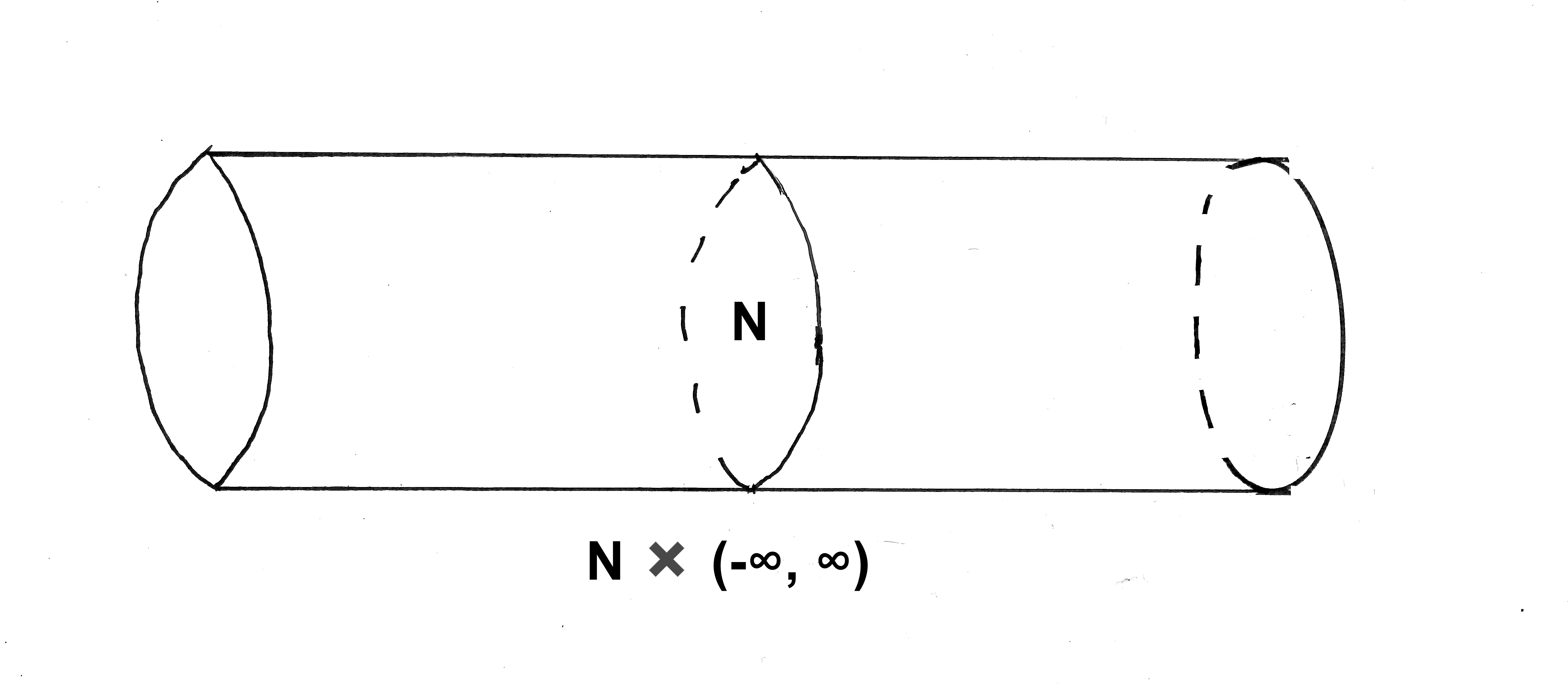}
\end{center}
\caption{The cylinder $N \times \R$}
\label{fig cylinder}
\end{figure}

In this section, we consider the setting of Subsection \ref{sec index thm}. In particular, $D$ and $\Phi$ are assumed to be of the forms \eqref{eq D D0} and \eqref{eq Phi Phi0}.

\subsection{An index on the cylinder}\label{sec ind cyl}

Let $S^{N \times \R}_{\pm} \to N \times \R$ be the pullbacks of the bundles of $S^{N}_{\pm} \to N$ defined in Subsection \ref{sec index thm} along the projection $N \times \R \to N$. They are Clifford modules, with Clifford actions
\[
\hat c(v , t) = c(v + t\hat n),
\]
for $v \in TN$ and $t \in \R$, where $c$ is the Clifford action by $TM$ on $S$ (which preserves $S^{N}_+$ as pointed out in Subsection \ref{sec index thm}), and $\hat n$ is the normal vector field to $N$ in the direction of $M_+$. Let $D^{S^{N\times R}_+}_0$ be the Dirac operator on $\Gamma^{\infty}(S^{N \times \R}_+)$ defined by this Clifford action, and the pullback to $N \times \R$ of the restriction to $N$ of the Clifford connection $\nabla^{S^N_+}$ used to define $D_0^{S^N_+}$.

Let $\chi \in C^{\infty}(\R)$ be an odd function such that $\chi(t) = t$ for all $t \geq 2$. We also denote its pullback to $N \times \R$ by $\chi$. Then pointwise multiplication by $\chi$ is an admissible endomorphism for $D^{S^{N\times \R}_+}_0$. 
Whenever a Dirac operator with a subscript $0$ is given, we will remove that subscript to denote the corresponding Dirac operator on two copies of the Clifford module in question, as in \eqref{eq D D0}. In the current setting, this gives us the Dirac operator
\[
D^{S^{N\times \R}_+} = 
\begin{pmatrix}
0 & D^{S^{N\times \R}_+}_0\\ D^{S^{N\times \R}_+}_0  & 0
\end{pmatrix}
\]
on $\Gamma^{\infty}(S^{N \times \R}_+  \oplus S^{N \times \R}_+ )$. We also consider the admissible endomorphism
\[
\chi^{N \times \R} = \begin{pmatrix}
0 & i\chi \\ -i \chi & 0
\end{pmatrix}
\]
of $S^{N \times \R}_+  \oplus S^{N \times \R}_+$.
\begin{proposition} \label{prop cylinder}
We have
\beq{eq cylinder}
\ind_G(D + \Phi) = \ind_G(D^{S^{N\times\R}_+} + \chi^{N \times \R}).
\eeq
\end{proposition}

The proof of Proposition \ref{prop cylinder} that we give below is an analogue of the proof of Theorem 5.4 in \cite{Cecchini16}. 
We give this proof in Subsections \ref{sec half cylin} and \ref{sec pf cyl}, referring to \cite{Cecchini16} for details in some places, and using results from \cite{Guo18} and from Section \ref{sec prop ind}.

\subsection{Attaching a half-cylinder} \label{sec half cylin}

Let $S_0^{N \times \R} \to N \times \R$ be the pullback of $S_0|_N \to N$. We choose $U$ small enough so that
 $S_0|_{U} \cong S^{N\times \R}_0|_U$. 

Because the sets $X_j$ are cocompact in Corollary \ref{cor 1234}, we initially compare the left-hand side of \eqref{eq cylinder} to an index on a manifold where only $M_+$ is replaced by a half-cylinder $N \times [1,\infty)$. To be more precise, $\ind_G(D + \Phi)$ is invariant under changes in the Riemannian metric on cocompact sets because the  Kasparov $(\C, C^*(G))$-cycles corresponding to two $G$-invariant Riemannian metrics differing only a cocompact set are homotopic by convexity of the space of $G$-invariant Riemannian metrics.
We choose a metric such that there is a neighbourhood $U$ of $N$ that is isometric to $N \times (1/4, 7/4)$ (see Figure \ref{fig M}). 
\begin{figure}
\begin{center}
\includegraphics[width=0.7\textwidth]{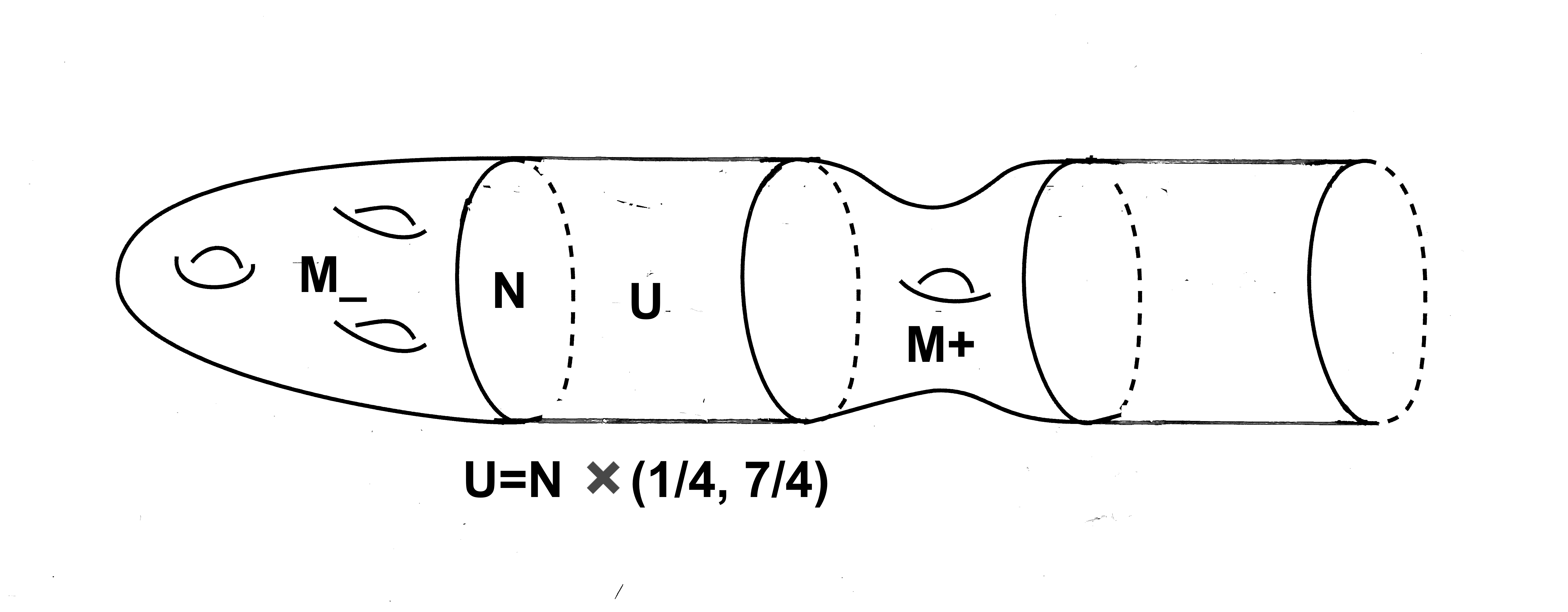}
\end{center}
\caption{The manifold $M$}
\label{fig M}
\end{figure}
 
 By Corollary \ref{cor cocpt change}, the index of $D+\Phi$ does not change if we change $\Phi_0$
  in a cocompact set. So we may assume that
 $\Phi_0$ is constant in the direction normal to $N$ inside $U$; i.e.\ for all $n \in N$ and $t \in (1/4, 7/4)$, $\Phi_0(n,t) = \Phi^N_0(n)$, for an endomorphism $\Phi^N_0$ of $S_0|_N$. We further choose $U$ such that a set $Z$ as in Definition \ref{def Phi} is contained in $M_- \setminus U$.
 
 Let $\nabla^{S_0^{N \times \R}}$ be the pullback of $\nabla^{S_0}|_{N}$ to a connection on $S_0^{N \times \R}$. We choose the Clifford connection $\nabla^{S_0}$ to define $D_0$ so that on $U$, it equals the restriction of $\nabla^{S_0^{N \times \R}}$  to $N \times (1/4, 7/4)$.

For this structure near $N$, we can form the Riemannian manifold $M_C \coloneqq  M_- \cup_N [1,\infty)$ (see Figure \ref{fig MC}), and define the Clifford module $S^C_0 \to M_C$ such that it equals $S_0$ on $M_-$ and $S^{N\times \R}_0$ on $N \times (1/4, \infty)$. Let $\nabla^{S_0^C}$ be the  Clifford connection on $S^C_0$ corresponding  to $\nabla^{S_0}$ on $M_-$ and to $\nabla^{S^{N\times \R}_0}$ on $N \times (1/4, \infty)$. Let $D^C_0$ be the resulting Dirac operator.
\begin{figure}
\begin{center}
\includegraphics[width=0.7\textwidth]{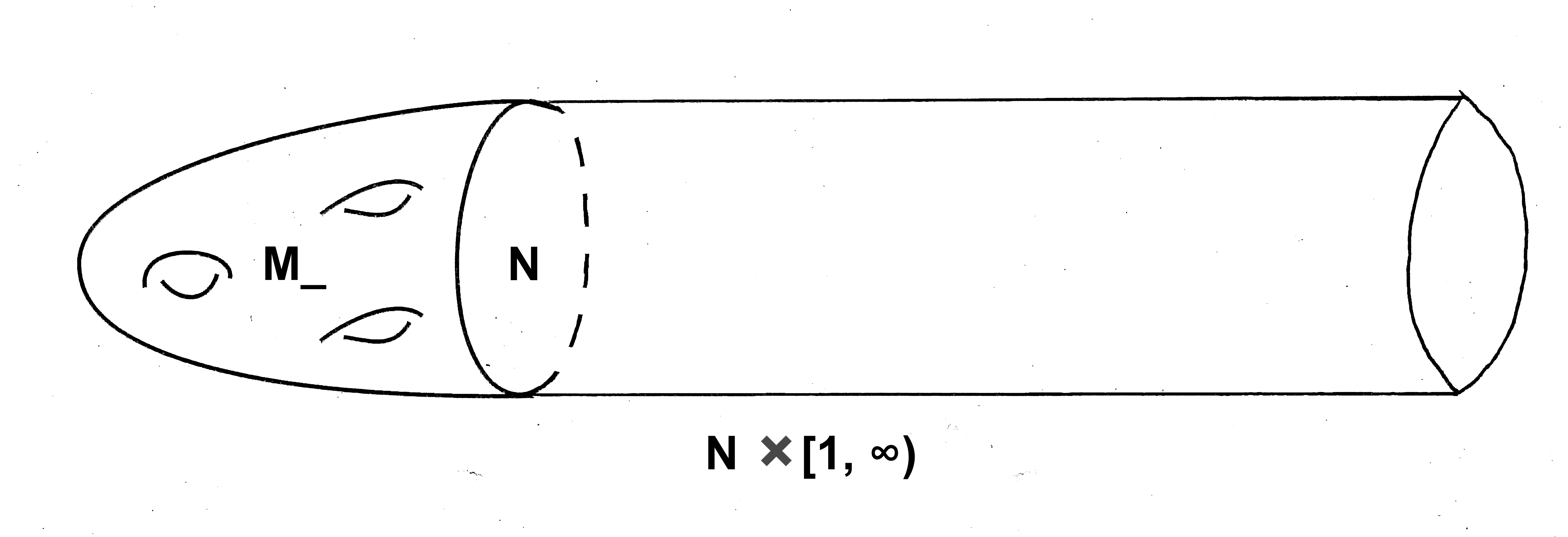}
\end{center}
\caption{The manifold $M_C$}
\label{fig MC}
\end{figure}

We define an endomorphism $\Phi^C_0$ of $S^C_0$ that is equal to $\Phi_0$ on $M_-$ and to the pullback of $\Phi^N_0$ on $N \times (1/4, \infty)$. 
Recall that by removing the subscript $0$ from $D^C_0$ we refer to the construction \eqref{eq D D0}. Similarly, when we remove the subscript $0$ from $\Phi^C_0$, we will be referring to the endomorphism $\Phi^C$ defined by $\Phi^C_0$ as in \eqref{eq Phi Phi0}.
Then Corollary \ref{cor 1234} immediately implies that
\beq{eq PhiC}
\ind_G(D + \Phi) = \ind_G(D^C + \Phi^C).
\eeq

The connection $\nabla^{S^{N\times \R}_0}$, and therefore the corresponding Dirac operator, does not preserve the decomposition $S^{N\times \R}_0 = S^{N\times \R}_+ \oplus  S^{N\times \R}_-$. With respect to this decomposition, that Dirac operator has the form
\beq{eq AB}
\begin{pmatrix}
D^{S^{N\times \R}_+}_0 & A \\ B & D^{S^{N\times \R}_-}_0,
\end{pmatrix}
\eeq
for vector bundle homomorphisms $A\colon S^{N\times \R}_- \to S^{N\times \R}_+$ and $B\colon  S^{N\times \R}_+ \to S^{N\times \R}_-$. (See Section 5.16 in \cite{Braverman18}, or use the fact that the difference of two connections is an endomorphism-valued one-form.) The Dirac operator $D^C_0$ equals this operator on $N \times (1/4, \infty)$. 
Let $\nabla^{S_{\pm}^{N \times \R}}$ be the pullback of $\nabla^{S_{\pm}^{N}}$ to a connection on $S_{\pm}^{N \times \R}$.
Consider a Clifford connection $\nabla^{S_0^C}$ on $S^C_0$ that is equal to the direct sum of $\nabla^{S_+^{N \times \R}}$ and $\nabla^{S_-^{N \times \R}}$ on $N \times (1/2, \infty)$ and to $\nabla^{S_0}$ on $M_- \setminus U$.
Then
the corresponding Dirac operator $\tilde D^C_0$
is equal to
\[
\begin{pmatrix}
D^{S^{N\times \R}_+}_0 & 0 \\ 0 & D^{S^{N\times \R}_-}_0
\end{pmatrix}
\]
on $N \times (1/2, \infty)$ and to $D_0$ on 
$M_- \setminus U$.
\begin{lemma} \label{lem DC tilde}
There exists a $\lambda \geq 1$ such that $\lambda \Phi^C$ is admissible for $\tilde D^C_0$. For such $\lambda$,
\beq{eq DC tilde}
\ind_G(D^C + \Phi^C) = \ind_G(\tilde D^C + \lambda \Phi^C).
\eeq
\end{lemma}
\begin{proof}
Existence of $\lambda$ with the desired property can be established as in the proof of Lemma 5.13 in \cite{Cecchini16}. The equality \eqref{eq DC tilde} can be proved via a linear homotopy; again see the proof of Lemma 5.13 in \cite{Cecchini16} for details, where references to Proposition 4.1 in \cite{Cecchini16} should be replaced by references to Proposition \ref{prop htp invar} in the current paper, and one uses Propositions \ref{prop E0} and \ref{prop:boundedhilbert} to check that the first condition in Proposition \ref{prop htp invar} is satisfied.

To be explicit, an application of Proposition \ref{prop htp invar} shows that 
\[
\ind_G(D^C + \lambda \Phi^C) =
\ind_G(\tilde D^C + \lambda \Phi^C).
\]
This follows from Proposition \ref{prop htp invar} by setting $P = D^C + \lambda \Phi^C$, $\Psi = \begin{pmatrix}
0 & A \\ B & 0,
\end{pmatrix}$, with $A$ and $B$ as in \eqref{eq AB}, $t_0 = -1$ and $t_1 = 0$. Note that although $\Psi$ is not a Callias-type potential here, Proposition \ref{prop htp invar} still applies, since $\Psi$ defines an element of $\mathcal{L}(\mathcal{E}^0)$ by Proposition \ref{prop E0} and an element of $\mathcal{L}(\mathcal{E}^1)$ by Proposition \ref{prop:boundedhilbert}. (The conditions of Proposition \ref{prop:boundedhilbert} hold because of the forms of $D^C$, $\Phi^C$ and $\Psi$.)
A more straightforward application of Proposition \ref{prop htp invar} yields $\ind_G(D^C +\Phi^C) = \ind_G(D^C + \lambda \Phi^C)$.
\end{proof}

Let $\chi \in C^{\infty}(\R)$ be an odd function such that 
\[
\chi(t) = 
\left\{
\begin{array}{ll}
 0 & \text{if $ -1/4 \leq t \leq1/4$};\\ 
1 & \text{if $3/4 \leq t \leq 3/2$};\\
t & \text{if $t\geq 2$}.
\end{array}
\right.
\]
(The property that $\chi$ is unbounded in a way that makes it proper is only used in the proof of Lemma \ref{lem diff cpt}; see Lemma \ref{lem ebert}.)
Such functions form a subset of the set of functions $\chi$ in Subsection \ref{sec ind cyl}, but Proposition \ref{prop cylinder} for general $\chi$ follows from the case for this class of functions, because the index on the right-hand side of \eqref{eq cylinder} does not change if we modify $\chi$ in a cocompact set. (And at any rate, to prove Theorem \ref{thm index}, we only need Proposition \ref{prop cylinder} to hold for one such function $\chi$.)

Let $\gamma^N$ be the grading operator on $S_0|_N$ that equals $\pm 1$ on $S^N_{\pm}$. Let $\gamma^{N\times \R}$ be its pullback to $S^{N \times \R}_0$. 
Let $\Phi_0^{\chi}$ be the endomorphism of $S_0^C$ equal to $\chi \gamma^{N \times \R}$ on $N \times (1/4, \infty)$ and equal to zero on the rest of $M_C$. 
\begin{lemma} \label{lem Phi chi}
We have
\[
\ind_G(\tilde D^C + \lambda \Phi^C) = \ind_G(\tilde D^C + \Phi^{\chi}).
\]
\end{lemma}
\begin{proof}
This can be proved via a linear homotopy between $\lambda \Phi^C$ and $\Phi^{\chi}$. The details are precisely as in the proof of Lemma 5.15 in \cite{Cecchini16}, with references to Propositions 4.1 and 5.9 in \cite{Cecchini16} replaced by references to Proposition \ref{prop htp invar} (combined with Propositions \ref{prop E0} and \ref{prop:boundedhilbert}) and Corollary \ref{cor 1234}, respectively, in the present paper. 
\end{proof}

\subsection{Proof of Proposition \ref{prop cylinder}} \label{sec pf cyl}

Let $M_-^-$ be the manifold $M_-$ with reversed orientation.
Form the manifold
\[
M_C^- \coloneqq  (N \times (-\infty, 1]) \cup_N M_-^-.
\]
See Figure \ref{fig MCminus}.
(The notation  is motivated by the fact that $M_C$ with reversed orientation is naturally equal to $(N \times (-\infty, -1]) \cup_N M_-^-$, which can be identified with $M_C^-$ via a shift over a distance $2$.)
\begin{figure}
\begin{center}
\includegraphics[width=0.7\textwidth]{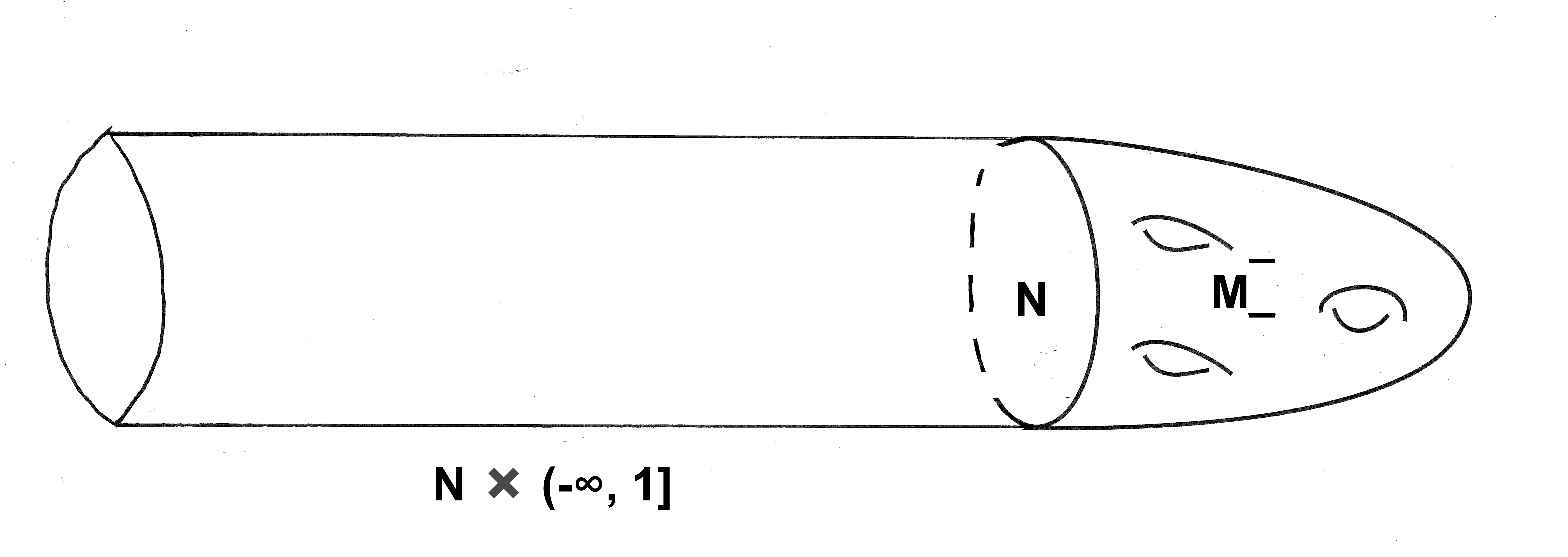}
\end{center}
\caption{The manifold $M_C^-$}
\label{fig MCminus}
\end{figure}

Let $S_0^- \to M^-_-$ be equal to the vector bundle $S_0|_{M_-}$, but with the opposite Clifford action (where $v \in TM^-_-$ acts as $c(-v)$).
Let $S_0^{C,-} \to M_C^-$ be the Clifford module that is equal to $S^{N \times \R}_0$ on $N \times (-\infty, 5/4]$ and to $S_0^-$ on $M^-_-$. From the Clifford connections $\nabla^{S^{N \times \R}_{\pm}}$ on $S^{N \times \R}_{\pm}$ and a Clifford connection $\nabla^{S_0^-}$  on $S_0^-$, construct a Clifford connection $\nabla^{S_0^{C, -}}$ on $S_0^{C,-}$ by 
\[
\nabla^{S_0^{C, -}} \coloneqq  \left\{ 
\begin{array}{ll}
\nabla^{S^{N \times \R}_+} \oplus \nabla^{S^{N \times \R}_-} & \text{on $N \times (-\infty, 5/4)$};
\\
\nabla^{S_0^-} & \text{on $M_-^-$}. 
\end{array}
\right.
\]
Using this connection, we obtain the Dirac operator $D^{C, -}_0$ on $\Gamma^{\infty}(S_0^{C,-})$. 
%
%
Then  $D^{C, -}_0$ equals $\tilde D^C_0$ on $N \times (1/2, 5/4)$.


The function $\chi$ equals $1$ on $(3/4, 5/4)$. 
Thus on this interval, both $\Phi_0^{\chi}$ and $\begin{pmatrix} \chi & 0 \\ 0 & -1\end{pmatrix}$ are equal to $\gamma^{N \times \R}$.
%
(Here we use $2\times 2$ matrix notation with respect to the decomposition $S_0^{N \times \R} = S_+^{N \times \R} \oplus S_-^{N \times \R}$.)
So we can define the endomorphism $\Phi^{C, -}_0$ of $S_0^{C, -}$ by setting it equal to  $ \Phi^{\chi}_0$ on $(N \times (3/4, 1]) \cup_N M_-^-$ and equal to
\beq{eq matrix chi -1}
\begin{pmatrix}
\chi & 0 \\ 0 & -1
\end{pmatrix}
\eeq
on $N \times (-\infty, 5/4)$, where $S_0^{C, -} = S_0^{N \times \R} = S_+^{N \times \R} \oplus S_-^{N \times \R}$. 

\begin{lemma} \label{lem DC-}
We have
\[
\ind_G(D^{C,-} + \Phi^{C, -}) = 0.
\]
\end{lemma}
\begin{proof}
Because $\chi = -1$ on $(-\infty, -3/4)$, we can define the endomorphism $\tilde \Phi^{C, -}_0$ of $S_0^{C, -}$ by setting it equal to $\Phi^{C, -}_0$ on $N \times (-\infty, -3/4)$ (where it equals \eqref{eq matrix chi -1}) and equal to $-1$ on $(N \times (-1, 1]) \cup_N M_-^-$. For that endomorphism, the estimate \eqref{eq est Phi} holds on all of $M$. Therefore by Lemma \ref{lem invtble},
\[
\ind_G(D^{C,-} + \tilde \Phi^{C, -}) = 0.
\]
The claim now follows from Corollary \ref{cor cocpt change}.
%
\end{proof}

\begin{proof}[Proof of Proposition \ref{prop cylinder}.]
Consider the cylinder $N \times \R$ as in Figure \ref{fig cylinder}. The data $(M_C, S^C_0, \Phi^{\chi}_0)$ and $(M_C^-, S_0^{C, -}, \Phi_0^{C,-})$ coincide in a neighbourhood of $N \times \{1\}$. 
By cutting along $N \times \{1\}$ and gluing, we obtain the corresponding data $(N \times \R, S_0^{N \times \R}, \Phi^{N \times \R}_0)$ 
and 
$(M_- \cup_N M_-^-, S_0^{M_- \cup_N M_-^-}, \Phi_0^{M_- \cup_N M_-^-})$.
See Figure \ref{fig M minus minus}.
To be explicit, 
\beq{eq Phi 0 N R}
\Phi_0^{N \times \R} = 
\begin{pmatrix}
\chi & 0 \\ 0 & -1
\end{pmatrix}
\eeq
on 
 $S^{N \times \R}_+ \oplus S^{N \times \R}_-$.

\begin{figure}
\begin{center}
\includegraphics[width=0.7\textwidth]{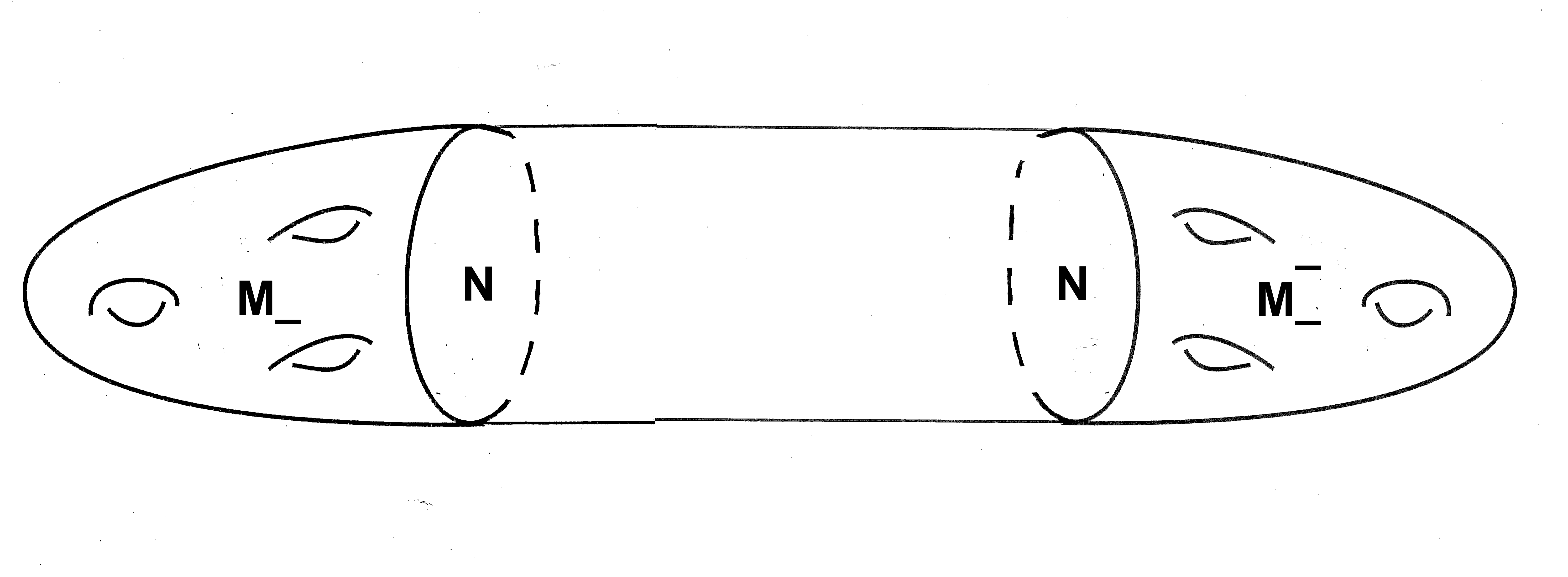}
\end{center}
\caption{The manifold $M_- \cup_N M_-^-$}
\label{fig M minus minus}
\end{figure}

Theorem \ref{thm rel index} implies that
\begin{multline*}
\ind_G(\tilde D^C + \Phi^{\chi}) + \ind_G(D^{C, -} + \Phi^{C, -}) \\
= \ind_G(D^{N \times \R} + \Phi^{N \times \R}) + \ind_G(D^{M_- \cup_N M_-^-} + \Phi^{M_- \cup_N M_-^-}).
\end{multline*}
By Lemmas \ref{lem cocpt} and \ref{lem DC-}, this implies that
\[
\ind_G(\tilde D^C + \Phi^{\chi}) 
= \ind_G(D^{N \times \R} + \Phi^{N \times \R}). 
\]

The connection $\tilde \nabla^{S_0^{N \times \R}}$ on $S_0^{N \times \R}$ obtained from cutting and gluing the connections $\nabla^{S_0^C}$ and $\nabla^{S_0^{C, -}}$ is the direct sum connection $\nabla^{S_+^{N \times \R}} \oplus \nabla^{S_+^{N \times \R}}$. So the corresponding Dirac operator $D_0^{S_0^{N \times \R}}$ equals
\[
D_0^{S_0^{N \times \R}} =
\begin{pmatrix}
D_0^{S_+^{N \times \R}} & 0 \\
0 & D_0^{S_+^{N \times \R}}
\end{pmatrix}.
\]
By the explicit form \eqref{eq Phi 0 N R} of $\Phi^{N \times \R}_0$, the operator $D_0^{S_0^{N \times \R}} \pm i \Phi^{N \times \R}_0$ on $\Gamma^{\infty}(S^{N \times \R}_0)$ is the direct sum of the operators $D_0^{S_+^{N \times \R}} \pm i \chi$ on  $\Gamma^{\infty}(S^{N \times \R}_+)$  and 
$D_0^{S_-^{N \times \R}} \mp i $ on  $\Gamma^{\infty}(S^{N \times \R}_-)$.
So
\[
\ind_G(D^{N \times \R} + \Phi^{N \times \R}) = \ind_G(D^{S^{N\times \R}_+} + \chi^{N \times \R} )+ 
\ind_G\left(D^{S^{N\times \R}_-} + \begin{pmatrix} 0 & -i \\ i & 0\end{pmatrix} \right).
\]
Lemma \ref{lem invtble} then implies that the second term on the right-hand side is zero, whence
\[
\ind_G(\tilde D^C + \Phi^{\chi}) = \ind_G(D^{S^{N\times \R}_+} + \chi^{N \times \R} ).
\]

The claim now follows from \eqref{eq PhiC} in conjunction with Lemmas \ref{lem DC tilde} and \ref{lem Phi chi}.
As pointed out above Lemma \ref{lem Phi chi}, the case for the class of functions $\chi$ we have used implies the case of the more general functions $\chi$ allowed in Proposition \ref{prop cylinder}.
\end{proof}

\subsection{Proof of Theorem \ref{thm index}}

Let $\cE_{N \times \R}$ be the Hilbert $C^*(G)$-module constructed from $\Gamma_c(S_+^{N \times \R})$ as in Subsection \ref{sec def index}. We write $D_{\chi}\coloneqq D^{S^{N \times \R}_+} + \chi^{N \times \R}$ for brevity.

\begin{lemma}\label{lem ebert}
For all $a>0$ the operator $(D_{\chi}^2 + a)^{-1}$ on $\cE_{N \times \R}$ is compact.
\end{lemma}
\begin{proof}
This is (a special case of) an analogue of Theorem 2.40 in \cite{Ebert18}. The proof proceeds in the same way, with the difference that the operator $D_{\chi}^2$ can only be bounded below by a function $h$ that is \emph{$G$-proper}, in the sense that the inverse image of a compact set is cocompact instead of compact. One chooses the bump functions $g_n$ in the proof of Theorem 2.40 in \cite{Ebert18} to be $G$-invariant and cocompactly supported. Theorem \ref{thm Rellich},  and Lemma 4.6(a) in \cite{Guo18}, imply that  $g_n (D_{\chi}^2+a)^{-1}$ is a compact operator on $\cE_{N \times \R}$. And as   in the proof of Theorem 2.40 in \cite{Ebert18}, one shows that $g_n (D_{\chi}^2+a)^{-1}$ converges to $ (D_{\chi}^2+a)^{-1}$ in the operator norm on $\calL(\cE_{N \times \R})$.
\end{proof}

We will use an analogue of Theorem 6.6 in \cite{Cecchini16}.
\begin{lemma} \label{lem diff cpt}
The operator
\beq{eq op diff}
D_{\chi}\bigl(D_{\chi}^2+f\bigr)^{-1/2} -
D_{\chi}\bigl( D_{\chi}^2+1\bigr)^{-1/2}
\eeq
lies in $\cK(\cE_{N \times \R})$.
\end{lemma}
\begin{proof}
By Proposition 4.12 in \cite{Guo18}, and as in (6.6) in \cite{Cecchini16}, the operator \eqref{eq op diff} equals
\beq{eq diff int}
\frac{2}{\pi} \int_0^{\infty} D_{\chi} (D_{\chi}^2 + f + \lambda^2)^{-1}(f-1) (D_{\chi}^2 + 1 + \lambda^2)^{-1}\, d\lambda.
\eeq
The operator $(D_{\chi}^2 + 1 + \lambda^2)^{-1}$ is compact by Lemma \ref{lem ebert}. Lemmas 4.6 and 4.8 in \cite{Guo18} imply that the integrand in \eqref{eq diff int} is bounded by $a(b+\lambda^2)^{-1}$ for certain $a,b>0$, so the integral converges in operator norm. It therefore defines a compact operator.
\end{proof}

Theorem \ref{thm index} follows from Proposition \ref{prop cylinder} and the following fact.
\begin{proposition}
In the setting of Proposition \ref{prop cylinder},
\[
\ind_G(D_{\chi}) = \ind_G(D^{S^N_+}). 
\]
\end{proposition}
\begin{proof}
Let $\cE_{N}$ be the Hilbert $C^*(G)$-module constructed from $\Gamma_c(S_+^{N})$ as in Subsection \ref{sec def index}.
 Then $\cE_{N \times \R} = \cE_{N} \otimes L^2(\R)$. So
Lemma \ref{lem diff cpt} implies that
$
\ind_G(D_{\chi}) 
$
is represented by the unbounded Kasparov cycle 
\beq{eq index cycle}
\left( 
\cE_N \otimes L^2(\R) \otimes \C^2, 
D_{\chi} \right).
%
\eeq

The Callias-type operator $D_{\chi}$ 
on 
\[
\Gamma^{\infty}(S^{N \times \R}_+ \oplus S^{N \times \R}_+) = \Gamma^{\infty}(S^{N}_+) \otimes C^{\infty}(\R) \otimes \C^2
\]
is equal to
\beq{eq DNR N R}
\begin{pmatrix}
0 & D^{S^N_+} \\ D^{S^N_+} & 0
\end{pmatrix}
\otimes 1_{C^{\infty}(\R)} + 
 \gamma_{S^{N}_+} \otimes
\begin{pmatrix}
0 & i\frac{d}{dt} \\  i\frac{d}{dt}  & 0
\end{pmatrix}
+
1_{\Gamma^{\infty}_c(S^N_+)}\otimes
\begin{pmatrix}
0 &  -i\chi \\ i\chi & 0
\end{pmatrix},
\eeq
where $\gamma_{S^N_+}$ is the grading operator on $S^N_+$, equal to $-i$ times Clifford multiplication by the unit normal vector field $\hat n$ on $N$ pointing into $M_+$ (so that $\gamma_{S^N_+} \otimes i\frac{d}{dt} = c(\hat n) \otimes \frac{d}{dt}$). 
%
%
%
Let $\cE_N^{\pm}$ be the even and odd-graded parts of $\cE_N$, and let $D^{S^N_+}_{\pm}$ be the restriction of $D^{S^N_+}$ to even and odd-graded sections of $S^N_+$, respectively.
With respect to the decomposition
\beq{eq ENR decomp}
\cE_N \otimes L^2(\R) \otimes \C^2 = 
(\cE_N^+ \otimes L^2(\R)) \oplus
(\cE_N^- \otimes L^2(\R)) \oplus
(\cE_N^+ \otimes L^2(\R)) \oplus
(\cE_N^- \otimes L^2(\R)), 
\eeq
the operator \eqref{eq DNR N R} equals
\beq{eq matrix Dchi}
D_{\chi} =
\begin{pmatrix}
0 & 0 & 1_{\cE_N^+} \otimes ( i\frac{d}{dt} -i\chi) & D^{S^N_+}_- \otimes 1_{L^2(\R)} \\
0 & 0 & D^{S^N_+}_+ \otimes 1_{L^2(\R)} &  1_{\cE_N^-} \otimes ( -i\frac{d}{dt} -i\chi) &  \\
1_{\cE_N^+} \otimes ( i\frac{d}{dt} +i\chi) & D^{S^N_+}_- \otimes 1_{L^2(\R)} & 0 & 0 \\
 D^{S^N_+}_+ \otimes 1_{L^2(\R)} & 1_{\cE_N^-} \otimes ( -i\frac{d}{dt} +i\chi) & 0 & 0 
\end{pmatrix}
\eeq

The kernel of $ i\frac{d}{dt} \pm i\chi$ in $C^{\infty}(\R)$ is one-dimensional, and spanned by the function
\[
f_{\pm}(t) = e^{\mp \int_0^t \chi(u)\, du}.
\]
By the properties of $\chi$, $f_+ \in L^2(\R)$, whereas $f_- \not \in L^2(\R)$. It follows that  $i\frac{d}{dt} - i\chi$ is invertible on the appropriate domain, while $ i\frac{d}{dt} + i\chi$ is zero on $\C f_+$ and invertible on $f_+^{\perp}$. 

Consider the submodules 
\[
\begin{split}
\cE_1 &\coloneqq  (\cE_N^+ \otimes \C f_+) \oplus 0 \oplus 0 \oplus  (\cE_N^- \otimes \C f_+);\\
\cE_2 \coloneqq  
\cE_1^{\perp} &=
(\cE_N^+ \otimes f_+^{\perp})  \oplus
(\cE_N^- \otimes L^2(\R)) \oplus
(\cE_N^+ \otimes L^2(\R)) \oplus
(\cE_N^- \otimes f_+^{\perp}) 
\end{split}
\]
of \eqref{eq ENR decomp}. These are preserved by the operator $D_{\chi}$. (For $\cE_1$, this is immediate from \eqref{eq matrix Dchi}; for $\cE_2$, this follows from the facts that $D_{\chi}$ is symmetric and preserves $\cE_1$.)
We find that the cycle \eqref{eq index cycle} decomposes as 
\beq{eq index cycle 2}
\left( 
\cE_1, 
\begin{pmatrix}
0 & 0 & 0 & D^{S^N_+}_- \otimes 1_{\C f_+} \\
0 & 0 & 0 & 0 \\
0 & 0 & 0 & 0 \\
D^{S^N_+}_+ \otimes 1_{\C f_+} & 0 & 0 & 0
\end{pmatrix}
\right)
\oplus (\cE_2, D_{\chi}|_{\cE_2}).
\eeq

The operator  $D_{\chi}|_{\cE_2}$ is essentially self-adjoint on the initial domain of compactly supported smooth sections by Proposition 5.5 in \cite{Guo18}, and its square has a positive lower bound in the Hilbert $C^*(G)$-module sense. Thus its self-adjoint closure is invertible, so that
the second term in \eqref{eq index cycle 2}
is homotopic to a degenerate cycle. The first term 
 represents $\ind_G(D^{S^N_+})$.
%
%
\end{proof}
\begin{remark}
A similar argument in the case where $G$ is trivial is hinted at below Lemma 4.1 in \cite{Kucerovsky01}.		
\end{remark}

\section{Proofs of results on positive scalar curvature} \label{sec obstr}

\subsection{Obstruction results} \label{sec pf obstr}

%

We now deduce Theorem \ref{thm obstr psc} from Theorem \ref{thm index}, and Corollary \ref{cor obstr psc} from Theorem \ref{thm obstr psc} and index theorems in \cite{HW2,  Wangwang, Wang14}.

\begin{proof}[Proof of Theorem \ref{thm obstr psc}]
This proof is an adaptation of the proof of Theorem 2.1 in \cite{Anghel93}.

First suppose that $M$ is odd-dimensional.
Let $\kappa$ denote scalar curvature. Let $K \subset \overline X$ be a cocompact subset of $M$ such that $N \subset K$, $\kappa > 0$ on $K$, and the distance from  $X \setminus K$ to $Y$ is positive. 
Let 
$\chi \in C^{\infty}(M)^G$ be a function such that $\chi(x) = 1$ for all $x \in Y$ and $\chi(x) = -1$ for all $x \in X \setminus K$.  Consider the operator $D$ as in \eqref{eq D D0}, where $D_0$ is the $\Spin$-Dirac operator on $M$, and the admissible endomorphism $\Phi$ as in \eqref{eq Phi Phi0}, where $\Phi_0$ is pointwise multiplication by $\chi$.
Then the set $M_-$ in Subsection \ref{sec index thm} can be chosen to be cocompact as required, and so that $N = N^- \cup H$, where $f|_{N^-} = -1$. In this setting,
\beq{eq SH}
S^{N}_+ = S_0|_H,
\eeq
 where $S_0 \to M$ is the spinor bundle. (This is consistent with Corollary \ref{cor Spinc}.)

For any $\lambda \in \R$, Lichnerowicz' formula implies that
\[
(D_0\pm i \lambda \chi)^2 = D_0^2 \pm i \lambda c(d\chi) + \lambda^2 \chi^2 
\geq \kappa/4 - \lambda \|d\chi\| + \lambda^2 \chi^2.
\]
On $M \setminus K$, the function on the right-hand side equals $\kappa/4 + \lambda^2 \geq \lambda^2$. Since $\kappa$ is $G$-invariant, and positive on the cocompact set $K$, it has a positive lower bound on that set. Further, $d\chi$ is $G$-invariant and cocompactly supported, hence bounded. So we can choose $\lambda>0$ small enough so that $\kappa/4 - \lambda \|d\chi\| > 0$ on $K$. It follows that $(D_0\pm i \lambda \chi)^2$ has a positive lower bound. This implies that $D + \Phi$ is invertible, so $\ind_G(D + \Phi) = 0$. The claim then follows by Theorem \ref{thm index} and \eqref{eq SH}.

If $M$ is even-dimensional, then the claim follows by applying the result in the odd-dimensional case to the manifold $M \times S^1$.
\end{proof}



\begin{proof}[Proof of Corollary \ref{cor obstr psc}]
This follows from Theorem \ref{thm obstr psc} and index formulas for traces defined by orbital integrals applied to $\ind_G(D^H)$. These index formulas are:
\begin{itemize}
\item Theorem 6.10 in \cite{Wang14} if  $G$ is any locally compact group and $g = e$;
\item Theorem 6.1 in \cite{Wangwang} if $G$ is discrete and finitely generated and $g$ is any element;
\item Proposition 4.11 in \cite{HW2} if $G$ is a connected semisimple Lie group and $g$ is a semisimple element.
\end{itemize}
These results imply that in the setting of Corollary \ref{cor obstr psc}, 
\[
\tau_g(\ind_G(D^H)) = \hat A_g(H),
\]
for a trace $\tau_g$.
As $\ind_G(D^H)$ is independent of the choice of Riemannian metric, so is $\hat A_g(H)$. Finally, vanishing of  $\ind_G(D^H)$ implies vanishing of $\hat A_g(H)$ for all $g$ as above.
\end{proof}

%
%
%
%

\subsection{Existence result}
\label{sec pf exist}

In the remainder of this section, we prove Theorem \ref{thm:noncptlawsonyauG} by generalising a construction by Lawson and Yau \cite{Lawson74}. We first prove in this subsection an extension of Theorem 3.8 in \cite{Lawson74}, namely Proposition \ref{prop:easyextension}.
	
	To prepare, let us recall the steps in the construction of Lawson and Yau's positive scalar curvature metrics \cite{Lawson74}, which we denote by $\tilde{g}_t$, on a compact manifold $N$. 
	
	Let $K$ be a compact Lie group acting on $N$.
%
Consider the principal $K$-bundle defined by the map
\[	
p\colon K\times N\rightarrow N,\qquad (k,y)\mapsto k^{-1}y.
\]
Take a $K$-invariant Riemannian metric $g_N$ on $N$. Let $b$ be a bi-invariant Riemannian metric on $K$.
Let $\hat{g}$ denote the lift of $g_N$ to the orthogonal complement to $\ker(Tp)$ in $T(K\times N)$ with respect to the product metric $b\oplus g_N$ on $K \times N$.
%

For each $t>0$, let $\hat{b}_{t^2}$ be the lift of the metric $t^2 b$ on $K$ to $TK \times N \subset T(K \times N)$. Then $$
	g_t\coloneqq \hat{g}\oplus\hat{b}_{t^2}$$
	is a Riemannian metric on the total space $K\times N$. One can check that, for each $t$, $g_t$ is invariant under the left $K$-action on $K\times N$ defined by
	\[
	l\cdot (k,y)=(kl^{-1},y).
	\]
	Thus $g_t$ descends, via the projection onto the second factor,
	$$\pi\colon K\times N\rightarrow N,\qquad(k,y)\mapsto y,$$
	to a $K$-invariant metric $\tilde{g}_t$ on $N$. Further, one sees that $\pi$ is a Riemannian submersion with respect to the metrics $g_t$ and $\tilde{g}_t$ on the total space and base respectively.

	The following proposition shows that, under the conditions stated,  for all sufficiently small $t$, $\tilde{g}_t$ has positive scalar curvature outside a neighbourhood of the fixed point set. It is an adaptation of the proof of Theorem 3.8 in \cite{Lawson74} to the more general setting when $N$ is non-compact but has $K$-bounded geometry; see Definition \ref{def K bdd geom}.
	\begin{proposition}
	\label{prop:easyextension}
		Let $N$ be a manifold with an action by a non-abelian, compact, connected Lie group $K$. Fix a bi-invariant metric on $K$. If $g_N$ is a $K$-invariant Riemannian metric on $N$ with $K$-bounded geometry, then for any neighbourhood $U$ of the fixed point set $N^K$, there exists $t_U>0$ such that for all $t\leq t_U$, the metric $\tilde{g}_t$ constructed above has uniform positive scalar curvature on $N\setminus U$.
	\end{proposition}
	\begin{proof}
		We will follow the steps in the proof of Theorem 3.8 in \cite{Lawson74} and show where the assumptions of bounded geometry and no shrinking orbits (Definition \ref{def:noshrinkingorbits}) are needed to obtain the conclusion.
		
		
		For each $y\in N\setminus N^K$, there is an orthogonal splitting $\mathfrak{k}=\mathfrak{k}_y\oplus\mathfrak{p}_y$, where $\mathfrak{k}_y$ is the Lie subalgebra of the isotropy subgroup $K_y$ of $y$. 
		The map $\varphi_y$ from \eqref{eq:differentiation} restricts to an injection on $\mathfrak{p}_y$. Denote the orthogonal complement of $\varphi_y(\mathfrak{p}_y)$ in $T_yN$ by $V_y$. Then
		$$T_{(e,y)}(K\times N)\cong\mathfrak{k}_y\oplus\mathfrak{p}_y\oplus\varphi_y(\mathfrak{p}_y)\oplus V_y.$$

		For each $y\in N\setminus N^K$, choose an orthonormal basis $\{ e_1(y),\ldots,e_{l_y}(y) \}$ of $\mathfrak{p}_y$ with respect to $b$ such that for all $j,k = 1, \ldots, l_y$,
\[		
g\bigl(\varphi_y(e_j(y)),\varphi_y(e_k(y)) \bigr)=\sigma^2_j(y)\delta_{jk}
\]
		for some continuous, positive functions $\sigma_j$. For each $j=1,\ldots,l_y$, define a function $\lambda_j\colon N\setminus N^K\to(0,\infty)$ by
		$$\lambda_j(y)\coloneqq\sigma_j(y)(1+\sigma_j(y)^2).$$ 		
		By the calculations in the proofs of Propositions 3.6 and 3.7 in \cite{Lawson74} and the assumption that $g_N$ has bounded geometry, for any neighbourhood $U$ of $N^K$, the scalar curvature of $\tilde{g}_t$ at any $y\in N\setminus U$ is bounded below by
		\begin{equation}
		\label{eq est off NK}
			\sum_{j,k=1}^{l_y}\frac{1}{t^2}\frac{\lambda_j(y)^2\lambda_k(y)^2}{(t^2+\lambda_j(y)^2)(t^2+\lambda_k(y)^2)}\norm{[e_j(y),e_k(y)]}_b^2+O(1)
		\end{equation}
		as $t\rightarrow 0$, where the $O(1)$ term is independent of $y$. 
	Since the $K$-action has no shrinking orbits with respect to $g_N$, there exists $c_U>0$ such that $\lambda_j(y)>c_U$ for each $y\in N\setminus U$ and $j=1,\ldots,l_y$. In particular, for $t\leq c_U$, the expression \eqref{eq est off NK} is bounded below by
		\begin{equation}
		\label{eq est off NK 2}
			\sum_{j,k=1}^{l_y}\frac{1}{4t^2}\norm{[e_j(y),e_k(y)]}_b^2+O(1).
		\end{equation}
Now, without loss of generality we may assume that $K=\SU(2)$ or $K = \SO(3)$, as any compact, connected, non-abelian Lie group has such a subgroup. Since $K$ has no subgroups of codimension $1$, we have $l_y=\dim\mathfrak{p}_y\geq 2$ at each $y\in N\setminus N^K$. For all $j$ and $k$, $\norm{[e_j(y),e_k(y)]}_b^2$ is 4 times the sectional curvature  of the plane spanned by $e_j(y)$ and $e_k(y)$ with respect to the metric $b$, and this is constant in $y$, and positive for $K = \SU(2)$ or $K = \SO(3)$. 
%
%
 Thus for any neighbourhood $U$ of $N^K$, there exists $t_U>0$ such that for all $t\leq t_U$, the expression \eqref{eq est off NK 2}, and hence also \eqref{eq est off NK}, is uniformly positive outside $U$. It follows that for all such $t$, the scalar curvature of $\tilde{g}_t$ is uniformly positive outside $U$.
	\end{proof}

We now deduce Theorem \ref{thm:noncptlawsonyauG} from the following noncompact generalisation of the main result in \cite{Lawson74}.
	\begin{theorem}
	\label{thm:noncptlawsonyau}
	Let $N$ be a manifold that admits an effective action by a compact, connected, non-abelian Lie group $K$, such that the fixed point set $N^K$ is compact. If there exists 
	 a $K$-invariant Riemannian metric on $N$ such that the $K$-action has $K$-bounded geometry, then $N$ admits a $K$-invariant metric with uniformly positive scalar curvature.
	\end{theorem}
	\begin{proof}
	Since $N^K$ is compact and the action is effective, by Section 4 of \cite{Lawson74} there exists $t_0>0$, a $K$-invariant neighbourhood $U$ of $K$ with compact closure, and a $K$-invariant Riemannian metric $g'$ on $N$ such that each metric $\tilde{g}_t'$, constructed from $g'$ as in subsection \ref{sec pf exist}, has positive scalar curvature on $U$ for $0<t<t_0$.

 Fix a bi-invariant metric $b$ on $K$. Let $g''$ be a $K$-invariant metric on $N$ for which the $K$-action has $K$-bounded geometry. Let $\{f_1,f_2\}$ be a smooth, $K$-invariant partition of unity on $N$ such that $f_1\equiv 1$ on $U$ and $f_1\equiv 0$ on $N\setminus U'$, where $U'$ is a relatively compact neighbourhood of $N^K$ containing the closure of $U$. 
Then 
\[
g_N\coloneqq f_1 g' +f_2 g''
\]	
	is a $K$-invariant Riemannian metric on $N$. Applying the prescription in Subsection \ref{sec pf exist} to $g_N$, we obtain a family $\{\tilde{g}_{t}\}_{t>0}$ of $K$-invariant metrics on $N$. We claim that for sufficiently small $t$, $\tilde{g}_t$ has uniformly positive scalar curvature on $N$.

To see this, let $\norm{\varphi}$ be the norm function  \eqref{eq norm phi} 
	 associated to the metric $g_N$. Since $g''$ and $g_N$ coincide on $N\setminus U'$, and $g''$ has $K$-bounded geometry, there exists $C_{U'}>0$ such that $\norm{\varphi}(y)\geq C_{U'}$ for all $y\in N\setminus U'$. One sees that $g_N$ has $K$-bounded geometry, and so by Proposition \ref{prop:easyextension}, there exists $t_1=(t_1)_{U}>0$ such that for all $t\leq t_1$, $\tilde{g}_{t}$ has uniformly positive scalar curvature on $N\setminus U$. 	 
	 It follows that for all $t\leq\min\{t_0,t_1\}$, $\tilde{g}_t$ has uniformly positive scalar curvature on $N$.
	\end{proof}
%
%

\begin{proof}[Proof of Theorem \ref{thm:noncptlawsonyauG}.]
In the setting of Theorem \ref{thm:noncptlawsonyauG}, Theorem \ref{thm:noncptlawsonyau} implies that $N$ admits a $K$-invariant metric with uniformly positive scalar curvature. By Theorem 4.6 in \cite{GHM1} (see also Theorem 58 in \cite{GMW}), this metric induces a $G$-invariant Riemannian  metric on $G \times_K N$ of uniformly positive scalar curvature.
\end{proof}

\section{Further applications of the Callias-type index theorem}  \label{sec appl}

We used Theorem \ref{thm index} to prove Theorem \ref{thm obstr psc}   in Subsection \ref{sec pf obstr}. We give some other applications of Theorem \ref{thm index} here.

\subsection{The image of the assembly map}

If $M/G$ is noncompact, and $G$ is not known to satisfy Baum--Connes surjectivity, then it is a priori unclear if $\ind_G(D + \Phi)$ lies in the image of the Baum--Connes assembly map \cite{Connes94}; see the question raised on page 3 of \cite{Guo18}. Theorem \ref{thm index} implies that this is in fact the case for $G$-Callias-type operators as defined above:\begin{corollary} \label{cor im mu}
The Callias-index of $D+\Phi$ lies in the image of the Baum--Connes assembly map.
\end{corollary}


\subsection{Cobordism invariance of the assembly map}

Theorem \ref{thm index} leads to a perspective on cobordism invariance of the analytic assembly map.
\begin{corollary}\label{cor cobord}
Let $X$ be an odd-dimensional Riemannian manifold with boundary $N$, on which $G$ acts properly and isometrically, preserving $N$, such that $M/G$ is compact. Suppose that a neighbourhood $U$ of $N$ is $G$-equivariantly isometric to $N \times [0,\varepsilon)$, for some $\varepsilon > 0$.
Let $S^X_0 \to X$ be a $G$-equivariant Clifford module, and consider the Clifford module $S^X_0|_N \to N$, graded by $i$ times Clifford multiplication by the inward-pointing unit normal.
Suppose that  $S^X_0|_U \cong S^X_0|_N \times [0,\varepsilon)$.
Let $D^N$ be the Dirac operator on $S^X_0|_N$ associated to a $G$-invariant Clifford connection $\nabla^{N}$ for the Clifford action by $TN$ on $S^X_0|_N$. Then $\ind_G(D^N) = 0$.
\end{corollary}
\begin{proof}
Form the manifold $M$ by attaching the cylinder $N \times (0,\infty)$ to $X$ along $U$. Extend  $S^X_0$ and the Clifford action to $M$ in the natural way. We write $S_0$ for the extension of $S^X_0$ to $M$. The connection $\nabla^N$ pulls back to a Clifford connection on $S_0|_{N \times (0,\infty)}$.
Because the $G$-invariant Clifford connections form an affine space, we can extend this pulled-back connection to all of $M$ using a $G$-invariant Clifford connection on $X \setminus N$ and a partition of unity. Let $D_0$ be the associated Dirac operator.

Let $\Phi_0$ be the identity endomorphism on $S_0$. Then $\Phi$, as in \eqref{eq Phi Phi0} is admissible for $D$ as in \eqref{eq D D0}, and  \eqref{eq est Phi} holds on all of $M$. By Lemma \ref{lem invtble}, this implies that $\ind_G(D + \Phi) = 0$. In this case, $S^N_+ = S^X_0|_N$, so by Theorem \ref{thm index}, $\ind_G(D^N) = 0$.
\end{proof}

\subsection{$\Spinc$-Dirac operators}

\begin{corollary} \label{cor Spinc}
Consider the setting of Theorem \ref{thm index}, and assume that 
 $M$ is odd-dimensional and $D_0$ is a $\Spinc$-Dirac operator. Then there is a union $N^+$ of connected components of $N$ and a $\Spinc$-structure on $N^+$ with spinor bundle $S_0|_{N^+}$ such that 
\beq{eq Spinc}
\ind_G(D+ \Phi) = \ind_G(D^{S_0|_{N^+}}) \quad \in K_0(C^*(G)),
\eeq
where $D^{S_0|_{N^+}}$ is a $\Spinc$-Dirac operator on the spinor bundle $S_0|_{N^+} \to N^+$.
If $N$ is connected, then $\ind_G(D+ \Phi) = 0$.
\end{corollary}
\begin{proof}
%
 Since $S_0$ is an irreducible Clifford module, and $S^N_+ \subset S_0|_N$  is invariant under the Clifford action of $TM|_N$, over each connected component $X$ of $N$ the bundle $S^N_+|_X$ is either zero or $S_0|_X$. Since $M$ is odd-dimensional, $S_0|_X$ is the spinor bundle of the $\Spinc$-structure of $X$ that it inherits from $M$. So \eqref{eq Spinc} follows.
%

If $N$ is connected, then we either have $N^+ = \emptyset$, in which case $\ind_G(D+ \Phi) = 0$ because $S^N_+$ is the zero bundle, or $N^+ = N$, in which case  \eqref{eq Spinc} and Lemma \ref{lem invtble}  imply that
\[
 \ind_G(D+ \Phi) =   \ind_G \left(D+ \begin{pmatrix} 0 & i \\ -i & 0 \end{pmatrix} \right)  = 0.
\]
\end{proof}

There is a converse to Corollary \ref{cor Spinc} in the following sense. Let $N^+ \subset N$ be any union of connected components; there is a finite number of such subsets of $N$ since $N/G$ is compact. We can define an admissible endomorphism $\Phi_0$ such that \eqref{eq Spinc} holds, by taking $\Phi_0$ to be multiplication by a $G$-invariant function on $M$ that equals $1$ on $N^+$ and $-1$ on $N \setminus N^+$, and is constant $1$ or $-1$ outside a cocompact set. Thus given any hypersurface $N$ bounding a cocompact set, and any set $N^+$ of connected components of $N$, we have an index
\beq{eq Spinc Callias}
\ind_G^{N^+}(D) := \ind_G(D + \Phi),
\eeq
with $\Phi$ and $\Phi_0$ related as in \eqref{eq Phi Phi0},
independent of the choice of $\Phi_0$ with the property that $\Phi_0$ is positive definite on $N^+$ and negative definite on $N \setminus N^+$.


Versions of the index \eqref{eq Spinc Callias} are sometimes used in applications of Callias-type index theorems to obstructions to positive scalar curvature for $\Spin$-manifolds, see  \cite{Anghel93,  Cecchini16} and the proof of Theorem \ref{thm obstr psc} in Subsection \ref{sec pf obstr}.

\subsection{Induction} \label{sec res ind}

Suppose that $G$ is an almost connected, reductive Lie group, and let $K<G$ be maximal compact.
In \cite{GMW, Hochs09, HW2} some results were proved relating $G$-equivariant indices to $K$-equivariant indices via Dirac induction. Such results allow one to deduce results in equivariant index theory for actions by noncompact groups from corresponding results for compact groups. This was applied to obtain results in geometric quantisation \cite{Hochs09, Hochs15, HM14} and geometry of group actions \cite{GMW, HM16}. 
Corollary \ref{cor Ind} below is a version of this idea for the index of Definition \ref{def index}. 

To state this corollary, we consider the setting of Subsection \ref{sec index thm}. Using Abels' slice theorem, we write $M = G \times_K Y$, for a $K$-invariant submanifold $Y \subset M$, and $S_0 = G \times_K S_0|_Y$. 
Let $\kg= \kk \oplus \kp$ be a Cartan decomposition. Then $TM \cong G \times_K(TY\oplus \kp)$. We assume that the $G$-invariant Riemannian metric on $M$ is induced by a $K$-invariant Riemannian metric on $Y$ and an $\Ad(K)$-invariant inner product on $\kp$ via this identification. 

We assume for simplicity that the adjoint representation $\Ad\colon K \to \SO(\kp)$ lifts to the double cover $\Spin(\kp)$ of $\SO(\kp)$. (This is true for a double cover of $G$.) Then the standard $\Spin$ representation $S_{\kp}$ of $\Spin(\kp)$ may be viewed as a representation of $K$. We assume that $S_0|_Y = S_0^Y \otimes S_{\kp}$ for a Clifford module $S_0^Y \to Y$. 
The Clifford action by $TM|_Y$ on $S_0|_Y$ equals $c_Y \otimes 1_{S_{\kp}} + 1_{S_0^Y} \otimes c_{\kp}$, for a Clifford action $c_Y$ by $TY$ on $S_0^Y$, and the Clifford action $c_{\kp}$ of $\kp$ on $S_{\kp}$. We choose the $G$-invariant connection $\nabla^{S_0}$ so that $\nabla^{S_0}|_Y = \nabla^{S_0^Y} \otimes 1_{S_{\kp}} + 1_{S_0^Y} \otimes \nabla^{S_\kp}$, for Clifford connections $\nabla^{S_0^Y}$ on $S_0^Y$ and $\nabla^{S_\kp}$ on $Y \times S_{\kp}$

Since $\Phi_0$ is $G$-equivariant, it is determined by its restriction to $Y$, which is a $K$-equivariant endomorphism of $S_0|_Y$.
We  assume that $\Phi_0|_Y = \Phi_0^Y \otimes 1_{\kp}$, for a $K$-equivariant endomorphism $\Phi_0^Y$ of $S_0^Y$. (What follows remains true if $\Phi_0 = \Phi_0^Y \otimes 1_{\kp} + 1\otimes \Phi_0^{\kp}$ for an $\Ad(K)$-invariant endomorphism of $S_{\kp}$, but this requires a small extra argument that we omit here.) 

Consider the Dirac operator $D_0^Y := c_Y \circ \nabla^{S_0^Y}$ on $\Gamma^{\infty}(S_0^Y)$. Form $D^Y$ from $D_0^Y$ as in \eqref{eq D D0} and $\Phi^Y$ from $\Phi_0^Y$ as in \eqref{eq Phi Phi0}. Let $R(K)$ be the representation ring of $K$ and
\begin{equation}
\label{eq DInd}
\DInd_K^G\colon R(K) \to K_*(C^*(G))
\end{equation}
be the Dirac induction map \cite{Connes94}.
\begin{corollary} \label{cor Ind}
The operator $D^Y + \Phi^Y$ is a $K$-equivariant Callias-type operator, and
\[
\ind_G(D + \Phi) = \DInd_K^G(\ind_K(D^Y + \Phi^Y)).
\]
\end{corollary}
\begin{proof}
Theorem \ref{thm index} implies that $\ind_G(D + \Phi) =\ind_G(D^{S^N_+})$. Write $Y^N := Y \cap N$, so that $N = G \times_K Y^N$. Then $Y^N$ is a compact manifold. Define $D^{S^{Y^N}_+}$ analogously to $D^Y$. The induction result for cocompact actions, Theorem, 4.5 in \cite{Hochs09}, Theorem 5.3 in \cite{HW2} or Theorem 46 in \cite{GMW}, implies that 
\[
\ind_G(D^{S^N_+}) = \DInd_K^G(\ind_K(D^{S^{Y^N}_+})).
\]
Another application of Theorem \ref{thm index}, now with $G$ replaced by $K$, or Theorem 1.5 in \cite{Anghel93} with a compact group action added, shows that $\ind_K(D^{S^{Y^N}_+} )= \ind_K(D^Y + \Phi^Y)$.
\end{proof}

\subsection{Callias quantisation commutes with reduction}

Theorem 3.11 in \cite{GHM2} is a \emph{quantisation commutes with reduction} result for the equivariant index of $\Spinc$-Callias-type operators. This result applies to reduction at the trivial representation of $G$; i.e.\ to an index defined in terms of $G$-invariant sections of $S$. Using Theorem \ref{thm index}, we can generalise this result to reduction at more general representations, or more precisely, at arbitrary generators of $K_0(C^*_r(G))$. Furthermore, this result is `exact' rather than asymptotic as Theorem 3.11 in \cite{GHM2}, in the sense that one does not need to consider high powers of a line bundle. 

In the setting of Subsection \ref{sec index thm}, we now assume that $M$ is odd-dimensional, and that $S_0$ is the spinor bundle for a $G$-equivariant $\Spinc$-structure. 
Let $D_0$ be the $\Spinc$-Dirac operator on  $\Gamma(S_0)$, defined by the Clifford connection corresponding to a connection $\nabla^L$ on the determinant line bundle $L$.

The \emph{$\Spinc$-moment map} associated to $\nabla^L$ is the map
$
\mu\colon M \to \kg^*
$
such that for all $X \in \kg$,
\[
2\pi i \langle \mu, X \rangle = \calL_X - \nabla^L_{X^M}, \quad \in \End(L) = C^{\infty}(M, \C),
\]
where $ \calL_X $ denotes the Lie derivative with respect to $X$, and $X^M$ is the vector field induced by $X$. The \emph{reduced space} at an element $\xi \in \kg^*$ is defined as $M_{\xi} := \mu^{-1}(\xi)/G_{\xi}$, where $G_{\xi}$ is the stabiliser of $\xi$. This reduced space is noncompact in general, and may not be smooth. But the reduced space $N_{\xi} := (\mu^{-1}(\xi)\cap N)/G_{\xi}$ is compact. It is not always a smooth manifold, but if it is, and $\xi \in \kk^*$, then we have an identification $N_{\xi} \cong Y^N_{\xi} := (\mu^{-1}(\xi) \cap Y \cap N)/K_{\xi}$, with $Y$ and $Y^N$ as in Subsection \ref{sec res ind}, including $\Spinc$-structures. See Propositions 3.13 and 3.14 in \cite{HM14}. Let $N^+ \subset N$
 be as in Corollary \ref{cor Spinc}. Then we similarly have $N^+_{\xi} \cong Y^{N^+}_{\xi} := (\mu^{-1}(\xi) \cap Y \cap N^+)/K_{\xi}$ in the smooth case, including $\Spinc$-structures.
 
There is a nontrivial way to define a \emph{$\Spinc$-quantisation} $Q^{\Spinc}(Y^{N^+}_{\xi}) \in \Z$, even when $Y^{N^+}_{\xi}$ is not smooth, described in detail in Section 5.1 of \cite{Paradan17}. Motivated by the identification $N^+_{\xi} \cong Y^{N^+}_{\xi}$ in the smooth case, we define $Q^{\Spinc}(N^+_{\xi}) := Q^{\Spinc}(Y^{N^+}_{\xi})$ for $\xi \in \kk^*$. 


Let $T<K$ be a maximal torus, and fix  a positive root system for $(K, T)$.
Let $V \in \hat K$ have highest weight $\lambda \in i\kt^* \hookrightarrow \kk^* \hookrightarrow \kg^*$. (The first inclusion is defined by $\kt^* \cong (\kk^*)^{\Ad^*(T)}$, the second by the Cartan decomposition.) Following \cite{Paradan17, Paradan18}, we call an element $\xi \in \kk^*$
 an \emph{ancestor} of $V$ if the coadjoint orbit $\Ad^*(K) \xi$  is admissible in the sense of \cite{Paradan18}, and  its $K$-equivariant $\Spinc$-quantisation is $V$. There exists a finite set $A(V)$ of ancestors representing all different such coadjoint orbits.
 
Let $C^*_r(G)$ be the reduced group $C^*$-algebra of $G$ 
and $\DInd_K^G$ the Dirac induction map \eqref{eq DInd}. By the Connes--Kasparov conjecture, proved in \cite{Chabert03, Lafforgue02b, Wassermann87},  the abelian group $K_*(C^*_r(G))$ is free, with  generators $\DInd_K^G[V]$, where $V$ runs over $\hat K$.

Recall the definition of the Callias index of $\Spinc$-Dirac operators \eqref{eq Spinc Callias}.
 \begin{corollary}[Callias quantisation commutes with reduction] \label{cor quant}
We have
\beq{eq quant}
\ind_G^{N^+}(D) = \bigoplus_{V \in \hat K} \sum_{\xi \in A(V)} Q^{\Spinc}(N^+_{\xi}) \DInd_K^G[V] \quad \in K_*(C^*_r(G)).
\eeq
\end{corollary}
\begin{proof}
By Corollary \ref{cor Spinc}, $\ind_G^{N^+}(D) = \ind_G(D^{S_0|_{N^+}})$, where now $D^{S_0|_{N^+}}$ is a $\Spinc$-Dirac operator on $N^+$. Theorem 4.6 in \cite{HM14} implies 
 that $\ind_G(D^{S_0|_{N^+}})$ equals
\[
\bigoplus_{V \in \hat K} \sum_{\xi \in A(V)} Q^{\Spinc}(N^+_{\xi}) \DInd_K^G[V]. 
\]
\end{proof}
\begin{remark}
In cases where $M_{\xi}$ is smooth and $N_{\xi}$ is a hypersurface in $M_{\xi}$, which is a transversality condition between $N$ and $\mu^{-1}(\xi)$, one can use Theorem 1.5 in  \cite{Anghel93} (more precisely, its special case for  $\Spinc$-Dirac operators, which is the non-equivariant case of  Corollary \ref{cor Spinc}) to express the $\Spinc$-quantisation $Q^{\Spinc}(N^+_{\xi})$ as the index of a Callias-type operator on $M_{\xi}$.
\end{remark}

\bibliographystyle{plain}

\bibliography{mybib}

\end{document}